\documentclass[article, 12pt]{amsart} 
\usepackage{amsmath,amssymb,amsthm,amsfonts}
\usepackage{mathtools}

\usepackage[shortlabels]{enumitem}
\usepackage[english]{babel}
\usepackage{microtype}
\usepackage[pdftex]{color}
\usepackage[bookmarks=true,hyperindex,pdftex,colorlinks, citecolor=MidnightBlue,linkcolor=MidnightBlue, urlcolor=MidnightBlue]{hyperref}
\usepackage[dvipsnames]{xcolor}

\usepackage[margin=2.5cm]{geometry}

\usepackage{colortbl}
\usepackage{todonotes}

\reversemarginpar
\renewcommand{\L}{\mathcal{L}}
\newcommand{\dom}{\mathrm{dom}}
\newcommand{\R}{\mathbb R}
\newcommand{\RP}{\mathbb{RP}}
\newcommand{\Pp}{\mathcal P}
\renewcommand{\P}{\mathcal P}
\newcommand{\C}{\mathcal C}
\newcommand{\Co}{\mathbb C}
\newcommand{\VP}{\mathcal P}
\newcommand{\Z}{\mathbb Z}
\newcommand{\E}{\mathbb E}

\newtheorem{thm}{Theorem}
\newtheorem{deff}{Definition}
\newtheorem{proposition}{Proposition}
\newtheorem{lemma}{Lemma}

\newtheorem{conj}{Conjecture}
\newtheorem{corollary}{Corollary}

\theoremstyle{remark} 
\newtheorem{remark}{\bf Remark}

\DeclareMathOperator{\supp}{supp}
\DeclareMathOperator{\vol}{vol}
\DeclareMathOperator{\inte}{int}

\DeclareMathOperator{\conv}{\rm conv}
\DeclareMathOperator{\pos}{\rm pos}

\newcommand{\eps}{\varepsilon}

\begin{document}

\title{ Volume Product}
\author{Matthieu Fradelizi}
\author{Mathieu Meyer}
 \author{Artem Zvavitch}\thanks{A.Z. is supported in part by the U.S. National Science
Foundation Grant DMS-1101636,  CNRS and U.S. National Science Foundation under Grant No. DMS-1929284 while A.Z. was in residence at the Institute for Computational and Experimental Research in Mathematics in Providence, RI, during the Harmonic Analysis and Convexity semester program.}


\date{January, 2023}

\begin{abstract}

 Our purpose here is to give an  overview of known results and open questions concerning the volume product $\VP(K)=\min_{z\in K}\vol(K)\vol((K-z)^*)$ of  a convex body $K$ in $\R^n$. We present a number of  upper and   lower bounds for $\VP(K)$, in particular, we discuss   the Mahler's conjecture on the lower bound of $\VP(K)$, which is still open. We also show connections of $\VP(K)$ with different parts of modern mathematics, including Geometric Number Theory, Convex Geometry, Analysis, Harmonic Analysis as well as Systolic and Symplectic Geometries  and Probability. 
\end{abstract}

\maketitle

\tableofcontents

\section{Introduction}

More or less attached to the name of Kurt Mahler (1903-1988), there are at least two celebrated unsolved problems:
 \begin{itemize}
\item  Lehmer's problem on algebraic numbers
\item  The  lower bound for  the volume product of convex bodies.
\end{itemize}
The celebrity of these two problems comes from the fact that they are both very natural and easy to state, but still unsolved and that for almost one century, a lot of mathematicians gave partial results, equivalent statements or many generalizations. There are still a lot of interesting attempts to resolve those problems which appear every now and then and produce connections of those questions to many areas of modern mathematics.

Although we shall be interested here in the second one, for the curiosity of the reader we summarize the first one. Let $\alpha$ be an algebraic integer and $P$ be the minimal polynomial of $\alpha$, that is the polynomial with integer coefficients with the smallest degree $d$ such that the coefficient of $x^d$ is $1$ and $\VP(\alpha)=0$. Let $\alpha_1=\alpha$ and $\alpha_2,\dots, \alpha_d\in \Co$ be the other roots of $P$. The {\it Mahler measure} of $\alpha$ is the number $M(\alpha):= \prod_{k=1}^d \max(|\alpha_k |,1)$.  By a classical result of Kronecker,  if 
$M(\alpha)=1$, then $\alpha$ is a root of unity. But how near can $\alpha$ be from $1$ when it is not a root of unity? Is there a constant $c>1$, independent of the degree of $\alpha$, such that $M(\alpha)>1$
implies that  $M(\alpha)>c$? Derrick Henry Lehmer conjectured  in 1933 \cite{Leh} that the answer to this question is positive (we note that for $c$ depending   on the degree of $\alpha$ a lot of estimates were given) and Mahler contributed to it, at least, by defining the measure to which his name was given \cite{Sm, VG}.

We shall be mainly concerned here with a second problem: Let $K$ be a convex  symmetric body in $\R^n$, which is the unit ball of a $n$-dimensional normed space $E$. Let $K^*$ be the polar body of $K$, which is the unit ball of $E^*$, the dual of $E$. What are the bounds for the {\it volume product} $\VP(E)= \VP(K):=\vol(K)\vol(K^*)$? It appears that the best upper bounds are known for a long time (Blaschke 1917 for $n=2,3$, \cite{Bl1}, Santal\'o 1949, $n\ge 4$ \cite{San2}), but to find the exact lower bounds is still an open conjecture, although the asymptotic behavior of $\min\{\VP(E); E \mbox{ $n$-dimensional normed space}\}$  was discovered almost 40 year ago by Bourgain and V. Milman \cite{BM}.  This problem has a lot of generalizations and specializations. One can ask a series of very natural questions including: What happens if $K$ is no longer centrally symmetric?  What happens for special classes of convex bodies? Is there a functional version of the volume product? What are the possible applications and connections inside and outside  convex geometry? We must also note that  a lot of new methods were used, in particular from functional analysis, harmonic analysis, topology, differential geometry and probability, to prove properties of the volume product and to attack different cases of this question.  We shall try here to explain just some of them and summarize  the others.

The paper is structured in the following way. We introduce the volume product and prove its basic properties in Section \ref{sec1980}. In Section \ref{secshadow}, we describe the methods of shadow systems which turn out to be essential in the study of the bounds for volume product. In Section \ref{BS}, we discuss the upper bound for the volume product - the Blaschke-Santal\'o inequality. We present different proofs, including a proof via Steiner symmetrizations and a harmonic analysis approach; we also discuss a number of extensions of this inequality. In Section \ref{MC}, we discuss the conjectured lower bound - the Mahler conjecture. We present a solution in a number of partial cases, including the case of zonoids, of unconditional bodies and of dimension $2$ and a very recent solution for  symmetric $3$-dimensional bodies.  We also present here an approach to stability results to both upper and lower bounds. Section \ref{AE} is dedicated to the asymptotic lower estimates for the volume product, in particular to the Bourgain-Milman inequality. Here, we extend our presentation to the harmonic analytic and complex analytic approach to the volume product. Section~\ref{func} is dedicated to the  functional inequalities related to the volume product with a special connection to transport inequalities. In section \ref{secMB}, we discuss the generalization of the volume product to the case of many functions and bodies. Finally, in section \ref{linsec}, we present a sample of connections of the bounds for volume product to other inequalities, including the slicing conjecture, Viterbo's conjecture, applications to the geometry of numbers and isosystolic inequalities. 

We refer the reader to \cite{AGM1, AGM2, Ga, Ga2, Gru, Kol, Pi, Sc, Tom} for many additional information on convex bodies, volume and mixed volume, duality and other core objects in analysis, geometry and convexity used in this survey. 

\vspace{.1in}
\noindent
{\bf Acknowledgments.}
We are grateful to Richard Gardner,  Dmitry Faifman, Dylan Langharst, Erwin Lutwak and Shlomo Reisner
 for many corrections,   valuable discussions and suggestions.

\subsection{Notations and results before 1980} \label{sec1980}

A convex body $K$ in $\R^n$ is a convex compact subset of $\R^n$ with nonempty interior denoted $\inte(K)$. We say that  $L\subset \R^n$ is centrally symmetric if $L=-L$.  Let $K$ be a convex body in $\R^n$ such that $0\in \inte(K)$; for $x\in \R^n$, we define $$\|x\|_K=\inf\{t>0; \,\, tx\in K\}$$
to be the {\it gauge} of $K$; in particular, when $K$ is a convex symmetric body, $x\mapsto \|x\|_K$ is the norm on $\R^n$ for which $K$ is the closed  unit ball. We endow $\R^n$ with its natural scalar product, denoted $\langle\ , \ \rangle,$ the associated Euclidean norm denoted $|\cdot |$; the Euclidean ball of radius one is denoted $B_2^n$. The canonical Lebesgue measure of a Borel set $A\subset \R^n$ is denoted by $\vol(A)$.  

Let $K$ be a convex body. If $0\in \inte(K)$, the {\it polar body} $K^*$ is defined by 
\begin{equation}\label{dualz}
K^*=\{y\in \R^n; \langle x,y\rangle \le 1 \hbox{ for all } x\in K\}.
\end{equation}
Then $K^*$ is also a convex body such that $0\in \inte(K^*)$  and if $T:\R^n\to \R^n$ is a linear isomorphism, one has $$\big(T(K)\big)^*= (T^*)^{-1}(K^*),$$ where $T^*$ is the adjoint of $T$.  More generally, for a convex body $K$ and  $z\in \inte(K)$, one defines {\it the polar body $K^z$ of $K$ with respect to $z$} by,
$$K^z=(K-z)^* +z.$$
The celebrated {\it bipolar theorem} asserts that  if $0\in \inte(K)$,   then $$(K^*)^*=K \hbox{ and consequently that } (K^z)^z=K$$ for any convex body $K$ and any $z\in \inte(K)$.
 Let $h_K(y)=\max_ {x\in K}\langle x,y\rangle$  be the {\it support  function} of $K$.  Note that $K^*=\{h_K\le 1\},$ i.e.  $h_K(x)=\|x\|_{K^*}$,  when  $0\in \inte(K).$ Moreover,
   if $  z\in \inte(K)$, $$K^z= z+\{y \in \R^n; h_K(y)-\langle z,y\rangle \le 1\}$$ and thus 
$$\vol(K^z)=\  \int_{K^*} \frac{1}{(1- \langle z,y\rangle)^{n+1}} dy.
$$
It follows   that the map  $z\mapsto\vol(K^z)$ is a strictly convex positive $C^{\infty}$ function on $\inte(K)$.

A small effort is enough to prove that $\vol(K^z)\to +\infty$ when $z$ approaches the boundary  of $K$.
Consider $\theta \in S^{n-1}$, 
by Brunn's theorem, the function $f_\theta:[-h_{K^*}(-\theta), h_{K^*}(\theta)]\to [0, \infty)$ defined by
$f_{\theta}(t):=\vol_{n-1} ( \{y\in K^*; \langle y,\theta\rangle=t\})$ satisfies that $f_\theta^{1/(n-1)}$ is concave. Hence, one has $f_\theta(t)\ge f_\theta(0) (1-h_{K^*}^{-1}(\theta)t )^{n-1}$ for $0\le t\le h_{K^*}(\theta)$. Let 
$
r_K(\theta) =\min\{a\ge 0: \theta \in aK\}
$
be the radial function of $K$. Then $r_K(\theta)=h_{K^*}^{-1}(\theta)$  
 and letting  $z=s\theta$ for $0\le s<r_K(\theta)$, we get
 $$
 \vol(K^z)
 =\int_{K^*} \frac{1}{(1- \langle z,y\rangle)^{n+1}} dy
 =\int_{-h_{K^*}(-\theta) }^{ h_{K^*}(\theta)} \frac{f_{\theta}(t)}{(1-st)^{n+1} }dt$$
 $$\ge f_{\theta}(0)\int_{0}^{ h_{K^*}(\theta) }\frac{ (1-th^{-1}_{K^*}(\theta) )^{n-1}  }{ (1-st)^{n+1} } dt
= \frac{f_{\theta}(0)}{n(r_{K}(\theta)-s) }
\ge\frac{ \min\limits_{\theta\in S^{n-1}} f_{\theta}(0)}{n(r_K(\theta)-s) }\to +\infty,$$ 
when $s\to r_{K}(\theta)$, that is $z\to \partial K$.

Consequently,  the function $z\mapsto \vol(K^z)$ reaches its minimum  on $\inte(K)$ at a unique point $s(K)$, called the {\it Santal\'o point} of $K$. Computing its differential, we see that $s(K)$ is characterized by the fact that the centroid (center of mass) of $K^{s(K)}$ is $s(K)$ (see \cite{MSW}). For $t>0$ big enough,
the sets $\{z\in \inte(K); \vol(K^z) \le t\}$, called {\it Santal\'o regions of $K$}, were studied in \cite{MW1} (see also \cite{MW2}).
\begin{deff}
 The Santal\'o point of a convex body $K$, denoted $s(K)$, is the unique point in $\inte(K)$ such that
$$\vol(K^{s(K)})=\min_{z\in K} \vol(K^z).$$

The volume product of $K$ is
 $$\VP(K):= \vol(K) \vol(K^{s(K)}).$$
 \end{deff}
\noindent We mention the following facts:
\begin{itemize}
 \item If $K$ is centrally symmetric, then so is $K^*$, and one has $s(K)=0= s(K^*)$ and $\VP(K)=\VP(K^*)$. One has always $\VP( K^{s(K) } )\le \VP(K)$ and when $K$ is not centrally symmetric, it may happen that $\VP( K^{s(K) } )<\VP(K)$.
 
 \item It is easy to see that $s(K)$ is the unique point of $\inte(K)$ such that $0$ is the center of mass of $\big(K-s(K)\big)^*$ or $s(K)$ is the center of mass of $K^{s(K)}$.

 \item The map  $K\mapsto \VP(K)$ is affine invariant, that is if $A:\R^n\to \R^n$ is a one-to-one affine transform, then $\VP(AK)=\VP(K)$. This indicates that if $E$ if an $n$-dimensional normed space with a   closed unit ball $B_E$ and if $\phi:E\to\R^n$ is a one-to-one linear mapping, then
   $\VP(E):=\VP(\phi(B_E))$ does not depend on $\phi$. In particular, this property makes  $\VP(E)$  to be an important tool in  the local theory of normed space (see \cite{Pi, LMi, Tom}).

\item Let $K$ be a convex body such that $0\in \inte(K)$ and let $E$ be a linear subspace of $\R^n$. Then, $K\cap E$ is a convex body in $E$, endowed with the scalar product inherited from the Euclidean structure of $\R^n$, and $(K\cap E)^*$ (with polarity inside $E$) can be identified with $P_E (K^*)$, where $P_E$ is the orthogonal projection from $\R^n$ onto $E$. Consequently,  when $K$ is centrally symmetric, $\VP(K\cap E)=\vol_E(K\cap E)\vol_E(P_E (K^*))$, where $\vol_E$ denotes the Lebesgue measure on $E$.
      \end{itemize}
       
 In view of  these facts, a natural question is to compute, for fixed $n$, an upper and and a lower bound of $\VP(K)$ for all convex bodies $K$ in $\R^n$.  The existence of these bounds follows from the affine invariance which allows to consider the bounds  of $K\mapsto \VP(K)$ on a compact subset of the set of all convex bodies endowed with the Hausdorff metric.  Indeed,  if $B_2^n$ is the Euclidean ball, by John's theorem, one may reduce to study $\VP(K)$ for $B_2^n\subset K\subset n B_2^n$ in the general case or 
   $B_2^n\subset K\subset \sqrt{n} B_2^n$ when $K$ is centrally symmetric, which gives already rough but concrete estimates for these bounds.

It seems that the first one who dealt with this problem was Wilhelm Blaschke (1885-1962), who proved that  for $n=2$ and $n=3$, $\VP(K) \le \VP({\mathcal E})$, where ${\mathcal E}$ is any ellipsoid  \cite{Bl1},  \cite{Bl2}. Then, Mahler gave exact lower bounds for $\VP(K)$ for $n=2$ both in the general case and in the centrally symmetric case \cite{Ma1, Ma2}. In 1947, Luis Santal\'o (1911-2001) \cite{San2} extended the results of Blaschke  to all $n$ with the same tools as him. The case of equality for the upper bound was first proved much later in 1978 by Petty \cite{Pe4} (the argument given in \cite{San2}  was not quite valid).
\begin{thm} {\bf (Blaschke-Santal\'o-Petty)} If $K\subset \R^n$ is a convex body,  then
\begin{equation}\label{BSE}
\VP(K)\le \VP(B_2^n),
\end{equation}
with 
    equality if and only if $K$ is an ellipsoid.
\end{thm}
   Bambah \cite{B} gave rough lower bounds for $\VP(K)$. Guggenheimer  \cite{Gu1, Gu2} believed at some moment that he had a complete proof of the exact lower bound 
   $\VP(K)\ge \VP([-1,1]^n)=\frac{4^n}{n!}$ for $K$ centrally symmetric, but it appeared that his proof was incorrect.
   This was the situation in the 80's, when  new insights on the problem were given by Saint-Raymond \cite{SR1}, Reisner \cite{Re1,Re2}, Gordon and Reisner 
   \cite{GR} and Bourgain-Milman \cite{BM}.
    
    We conclude this section by an important tool in this theory:
 \subsection{Shadow systems}\label{secshadow}

 \begin{deff}\label{defshadow}  A {\bf shadow system} along a direction $\theta \in S^{n-1}$ is a family of convex sets $K_t \subset \R^n$ 
which are defined by
$K_t = \conv( \{x+ta(x)\theta; x\in M\} )$
where $M$ is a bounded subset in $\R^n$,  $a:M\to \R$ is a bounded function   and $t\in I$, an interval of $\R$ (and where $\conv(A)$  denotes the closed convex hull of a set $A \subset \R^n)$.
\end{deff}
     It  may  be observed that the classical Steiner-symmetrization can be seen as a shadow system such that the volume of $K_t$ remains constant (see Remark~\ref{rk:steiner-shadow} below). The notion of shadow system was introduced by Rogers and Shephard \cite{RS,Sh2} and  can be explained via an idea of  Shephard in \cite{Sh2}, 
     who pointed out that a shadow system  of convex bodies in $\R^n$  can be seen as a family of
projections of a $(n+1)$-dimensional convex set  on some $n$-dimensional subspace of $\R^{n+1}$. Namely, let $e_1, e_2, . . . , e_{n+1}$ be an
orthonormal basis of $\R^{n+1}$. Let $M$ be as before be a bounded subset of $\R^n$ (i.e. $M$ is contained in the linear span of $e_1, e_2, . . . , e_{n}$), let $a:M\to \R$ be a bounded function, $\theta\in S^{n-1}$  and $I$ an interval of $\R$.  We define a convex set $\tilde{K}\subset \R^{n+1}$ be the $(n+1)$-dimensional 
by
$$
\tilde{K}= \conv\{x + a(x)e_{n+1}: x \in M\}.
$$ 
For each $t \in I$, let $P_t:\R^{n+1}\to \R^n$ be the  projection  from $\R^{n+1}$ onto $\R^n$ along the
direction $e_{n+1} - t \theta$. Then, $$P_t(\tilde{K})=
 \conv( \{x+ta(x)\theta; x\in M\} )= K_t.$$ This interpretation permits to see that 
$t\mapsto \vol(K_t)$ is a convex function \cite{RS}. This result
is a powerful tool for obtaining geometric inequalities of isoperimetric type.
   
The following theorem connects shadow systems with the volume product. It was proved by Campi and Gronchi \cite{CG1} when the bodies $K_t$ are centrally symmetric  and by Meyer and Reisner in the general case \cite{MR3} (see also  \cite{FMZ}).
\begin{thm}\label{shadow}
    Let $(K_t)_{t\in(a,b)}$ be a shadow system of convex bodies in $\R^n$. Then the function $t\mapsto \vol \big((K_t)^{s(K_t)}\big)^{-1}$ is convex on $(a,b)$.
   \end{thm}
With the previous notations, if the $K_t$ are centrally symmetric, then $s(K_t)=0$  and thus 
$$
(K_t)^{s(K_t)}=K_t^*=(P_t(\tilde{K}))^*.
$$
As it was discussed above, the polar of the orthogonal projection of a convex body on a subspace $E$ is the section of its polar by $E$. But here $P_t$ is not an orthogonal projection, and we get $$(P_t(\tilde{K}))^*=P_{e_{n+1}^\perp}(\tilde{K}^*\cap (e_{n+1} - t \theta)^\bot),$$ where $P_{e_{n+1}^\perp}$ is the orthogonal projection from $\R^{n+1}$ onto $\R^n=e_{n+1}^\bot$. 
 In that context, Campi-Gronchi's theorem appears as another formulation, with a new proof, of Busemann's theorem \cite{Bu}  about the central hyperplane sections of a centrally symmetric body (see also \cite{MR5}).
This point of view was put forward in \cite{CFPP}, where such properties  were also generalized to more general measures than Lebesgue measure.

An important component related to Theorem \ref{shadow}, which is proved
in \cite[Proposition 7]{MR3} (see also  \cite{MR4}), is
the case when both $\vol(K_t)$ and $\vol((K_t)^{s(K_t)})^{-1}$ are affine functions of $t\in (a,b)$. In this case, all
the bodies $K_t$ are affine images of each other by special affine transformations. This
has been useful in identifying the cases of equality in inequalities involving volume
products, as well as in the proof of the results of \cite{MR4} and \cite{FMZ}. The proof of this component, that involves some ODE, was extended in \cite[Proposition 6.1]{MY} to  generalize this  result .

\section{The Blaschke-Santal\'o inequality}\label{BS}

The original proofs of the Blaschke-Santal\'o inequality \cite{Bl1, San1, San2, Leic1} were based  on the affine isoperimetric inequality. We show now how this  inequality implies the Blaschke-Santal\'o inequality and how conversely the Blaschke-Santal\'o inequality implies it. We refer to Section 10 in \cite{Sc} and to \cite{SW, Leic2} for details about the tools used in this section. If $M$ is a smooth convex body with positive curvature everywhere, its affine surface area $\mathcal{A}(M)  $ is defined by 
$$\mathcal{A}(M)= \int_{S^{n-1}}  f_M(u)^{\frac{n}{n+1}} du,$$
where $f_M: S^{n-1}\to \R_+$ is the curvature function, i.e. the density of the
measure $\sigma_M$ on $S^{n-1}$ with respect to the Haar measure on $S^{n-1}$, where for a Borel subset $A$ of $S^{n-1}$, $\sigma_M(A)$ is the $(n-1)$-Hausdorff measure of the set of all points in $\partial M$ such that their unit normal to $M$ is in $A$.
 The {\it affine isoperimetric inequality} says that at a fixed volume, ellipsoids have the largest affine surface area among all convex bodies with
positive continuous curvature. In other words,
\begin{equation}\label{affine} \mathcal{A}(M)^{n+1} \le n^{n+1} v_n^2\vol(M)^{n-1},
\end{equation}
where $v_n=\vol(B_2^n)$. Let $L$ be another convex body, H\"older's inequality, one has 
\begin{align}\label{Holder}\mathcal{A}(M)^{n+1}&\le  \left(\int_{S^{n-1}} 
 h_L(u)f_M(u) du\right)^n
\int_{S^{n-1}} h_L^{-n}(u)du \nonumber\\&=n^{n+1}V(M[n-1],L)^n \vol(L^*),
\end{align}
where $V(M[n-1],L)=V(M[n-1],L[1])=\frac{1}{n} \int_{S_{n-1}} h_L(u)f_M(u) du$ is a mixed volume of $M$ and $L$, which can also be defined by the formula, for $t\ge 0$,   $$\vol(M+tL)=\sum_{k=1}^n \binom{n}{k}V(M[n-k],L[k])t^k.$$
We refer to \cite{Sc} for precise definition of properties  mixed volumes. Using inequality (\ref{affine}) and the Minkowski first inequality
$$
V(M[n-1],L)^n\ge \vol(M)^{n-1} \vol(L),
$$ one gets
\begin{equation}\label{facile}\mathcal{A}(M)^{n+1}\le n^{n+1}v_n^2\vol(M)^{n-1}\le n^{n+1}v_n^2 V(M[n-1],L)^n \vol(L)^{-1}.
\end{equation}
Now given a convex body $K$, let   $s=s(K)$ be  its Santal\'o point; then $0$ is the centroid of $K-s$ so that 
$$
\int_{S^{n-1}} uh_{K-s}(u)^{-(n+1)}du=0,
$$ 
and thus by Minkowski’s existence theorem (see Section 8.2 \cite{Sc}), there exists a convex body $M$
such that $f_M= h_{K-s}^{-(n+1)}$. Set  $L=K-s$, then there is equality in (\ref{Holder}) so that by (\ref{facile})
\begin{align*}
n^{n+1}V(M[n-1],K-s)^n  \vol((K-s)^*)&=\mathcal{A}(M)^{n+1} \\&\le n^{n+1}v_n^2 V(M[n-1],K)^n \vol(K-s)^{-1},
\end{align*}
which gives the Blaschke-Santal\'o inequality \eqref{BSE}.

Conversely, let  $M$ be a convex body with positive curvature, and suppose that $0$  is the Santal\'o point of $M$ and that Blaschke-Santal\'o inequality holds for $M$. By (\ref{Holder}) with $L=M$, we get
$$ \mathcal{A}(M)^{n+1}\le  n^{n+1}V(M[n-1],M)^n \vol(M^*)
\le n^{n+1}v_n^2 \vol(M)^{n-1},$$
which is  the affine isoperimetric inequality.

For examples of other results of this type and relations between affine surface area and the volume product, see Petty \cite{Pe3, Pe4}, Lutwak  \cite{Lu1}, Li, Sch\"utt and Werner and \cite{LSW}, Nasz\'odi, Nazarov and  Ryabogin \cite{NNR} and Hug \cite{Hu}, who gave a proof of the affine isoperimetric inequality using Steiner symmetrization and studied the cases of equality.

\subsection{A proof of the Blaschke-Santal\'o inequality  in the centrally symmetric case}\label{steiner-mp} In \cite{MP1, MP2}, Meyer and Pajor  used the classical Steiner symmetrization for giving a proof  which we sketch in this section. Various symmetrizations of sections appeared also in \cite{SR1} and \cite{Ba1}.
\vskip 2mm
\noindent{\it Step 1:} We prove first that if $H$ is a linear  hyperplane of $\R^n$ and $S_H K$ is the {\it Steiner symmetral of $K$ with 
 respect  to $H$}, as it will be defined below, then
  $$
  \vol(K^*) \le \vol \big( (S_HK)^*\big).
  $$
  To simplify notations, suppose that $H= \R^{n-1}\times \{0\}$; as before, let $P_H:\R^n\mapsto H$ be the orthogonal projection onto $H$. Then  $K$  may  be described as follows:
 $$K=\{x+se_n;\ x\in P_HK, s\in I(x)\} $$ where  for $x\in P_HK$, $I(x)=\{s\in\R; x+se_n\in K\}$ is  a closed interval.
  The Steiner symmetral of $K$  with respect to $H$, defined as 
  $$S_HK=\left\{x+se_n; x\in PK,s\in \frac{I(x)-I(x)}{2}\right\} $$
is  a convex body symmetric with respect to $H$, such that
$$\vol(K)=\vol(S_H K).$$
  For  $t\in \R$,  let $K^*(t):=\{y\in H; y+te_n \in K^*\}$  be the section of $K^*$ by the hyperplane $\{ x_n =t \}$ and $J=\{t\in \R;  K^*(t)\not=\emptyset\}$. Then
$$K^* =\{ y+t e_n;  t\in J, y\in K^*(t) \}.$$
By the symmetry of $K^*$,  one has $K(-x)=-K(x) $   for every  $x\in P_HK$,  so that for every $t\in J$, $y_1\in K^*(t) $, $y_2\in K^*(-t) $  and  $s_1,s_2\in K(x)$, one has:
$$
\langle x,y_1\rangle+s_1 t  \le 1\hbox{ and } \langle x,y_2\rangle -s_2t  \le 1.
$$
Adding these two inequalities, we get that for every $x\in P_HK$,  $s=\frac{s_1-s_2}{2}\in \frac{1}{2} (I(X)-I(X))$,  $t\in J$ and $y=\frac{y_1+y_2}{2}\in
 \frac{1}{2}\big(K^*(t)+ K^*(-t)\big)$,  one has
 $$\langle x,y\rangle+st \le 1.$$
 Thus  for every $t\in J$,  
 \begin{equation}\label{steiner-inclusion}
 \frac{1}{2}\big( K^*(t) +K^*(-t) \big)\subset (S_H K)^*(t).
\end{equation}
By the symmetry  of $K$, one has  $K^*(-t) = -K^*(t)$. It follows from Brunn-Minkowski's theorem that  
$\vol_{n-1}\big( (S_H K)^*(t)\big) \ge \vol_{n-1}\big(K^*(t) \big)$ and, integrating in $t$, we get that 
$$\vol\big((S_HK)^*\big)= \int \vol_{n-1}\big( (S_H K)^*(t)\big) dt \ge 
\int \vol_{n-1}(K^*(t) ) dt=\vol(K^*).$$
 One get thus that $\VP(S_H K)\ge \VP(K)$.

\noindent {\it Step 2:}  It is well known that there exists a sequence of hyperplanes $(H_m)_m$ such that if $K_0=K$ and for $m\ge 1$,  $K_m:=S_{H_m}K_{m-1}$, then the sequence $(K_m)_m$ converges to the Euclidean ball $R_KB_2^n$ with same volume as $K$ (see for example Section 10.3 in \cite{Sc}). Thus  $$\VP(K)\le \VP(K_{n-1})\le \VP( K_n)\le \VP(R_KB_2^n)=\P(B_2^n).$$
 
 The case of equality in Blaschke-Santal\'o inequality was first proved by Petty \cite{Pe4}, using sharp differential geometry arguments. When $K$ is centrally symmetric, a new elementary proof was given by Saint Raymond \cite{SR1}, using Minkowski symmetrization of the hyperplane sections (see also \cite{Ba1}). These ideas were generalized by Meyer-Pajor  \cite{MP2} to give an elementary proof for the general case, and  a somewhat stronger result, also based on Steiner's symmetrizations:
\begin{thm}{\bf (Meyer-Pajor \cite{MP2}) } For every convex body $K$  and every hyperplane $H$ separating $\R^n$ into two closed half space $H_+$ and $H_-$, denoting $\lambda=\frac{\vol( H_+\cap K)}{\vol(K)}$, there exists $z\in H$  such that 
 $\vol(K)\vol(K^z)\le \frac{\VP({\mathcal E})}{ 4\lambda(1-\lambda)}.$
\end{thm}

\begin{remark}\label{rk:steiner-shadow}  Notice
  that the Steiner's symmetral of a convex body $K$ with respect to a direction $u\in S^{n-1}$ can be written as  a part of a shadow system in the following way:  for $y\in P_{u^\perp}K$, let $I(y)=\{s\in\R; y+su\in K\}$. For $t\in[-1,1]$, let 
  \[
  K_t=\left\{y+su; s\in \frac{1+t}{2}I(y)-\frac{1-t}{2}I(y)\right\}.
  \]
  Then  $K_1=K$, $K_0=S_{u^\perp}K$ and, for every $t\in[-1,1]$, $K_{-t}$ is the reflection of $K_t$ with respect to the hyperplane $u^\perp$. This implies that the function $t\mapsto\P(K_t)$ is even. Moreover, since $\vol(K_t)=\vol(K)$ is constant for $t\in[-1,1]$,  using Theorem \ref{shadow}, the function $t\mapsto\P(K_t)^{-1}$ is convex. It follows that $\P(K_t)$ is maximized at $0$, which proves that the volume product is non-decreasing by Steiner symmetrization for any convex body, recovering Meyer-Pajor's result \cite{MP2}. Using again an appropriate sequence of Steiner symmetrizations, this gives the general Blaschke-Santal\'o inequality. The preceding observation is due to \cite{MR3}. 
 \end{remark}

\subsection{An harmonic analysis proof of the Blaschke-Santal\'o inequality}  We follow the work of Bianchi and  Kelli \cite{BK}.
Harmonic analysis plays an essential role in the study of duality and the volume product.  The main idea was discovered by Nazarov \cite{Na}, who used it to provide a proof of the Bourgain-Milman inequality \cite{BM} and was adopted by Bianchi and Kelly to give a very elegant proof of the  Blaschke-Santal\'o inequality. We  refer to \cite{StWe}  and \cite{Ho} for basic facts in harmonic analysis. 

Let  $K$  be a convex symmetric body  in $\R^n$. Let $F \in L_2(\R^n)$ such that its  Fourier transform $\widehat{F}=0$ a.e. on $\R^n\setminus K$. Then 
$F$ is the restriction to $\R^n$ of  the entire function still denoted $F$ defined by: 
$F(z)=\int_{K} e^{2\pi i \langle z,  \xi \rangle} \widehat{F}(\xi) d\xi \hbox{ for }z \in {\mathbb C}^n, $
which satisfies the following inequality giving a first hint on why the theory is so useful:
$$
|F(i y)|\!
=\!\left | \int_K e^{-2\pi  \langle y, \xi \rangle} \widehat{F}(\xi) d\xi\right |  
\le \int_K e^{2\pi  \sup\limits_ {\xi  \in K} |\langle \xi, y \rangle|} | \widehat{F}(\xi)| d\xi = e^{2 \pi \|y\|_{K^*} }\!\!\!  \int_K | \widehat{F}(\xi)| d\xi.$$
Thus  for some $C>0$, one has $|F(i y)| \le C e^{2\pi \|y\|_{K^*} }$ for all $y\in \R^n$, i.e. $F$ is {\it of exponential type $K^*$}. This fact  is the elementary part of the following classical theorem:
\begin{thm}\label{PW}(Paley-Wiener)
Let $F \in L_2(\R^n)$ and $K$ be a convex symmetric body. Then the following are equivalent:
 \begin{itemize}
\item  $F$ is the restriction to $\R^n$ of an entire function of exponential type $K^*$.
\item  The support of $\widehat{F}$ is contained in $K$.
\end{itemize}
\end{thm}

Now we are ready to present  the
\vskip 2mm\noindent
\noindent {\bf Proof of the Blaschke-Santal\'o inequality adapted from Bianchi and Kelly \cite{BK}.}
Let $F=\frac{1}{\vol( K)}  \widehat{ {\bf 1}_K}$, where ${\bf 1}_K$ is the characteristic function of $K$:
$$
{\bf 1}_K(x)=\begin{cases}
1, \mbox{ for } x \in K\\
0, \mbox{ for } x \not \in K.
\end{cases}
$$Then  $F(0)=1$, $F\in L_2(\R^n)$, $F$ is continuous and even, and can be extended as an entire function on $ {\mathbb C}^n$, still denoted $F$, 
as 
$$
F(z_1,\dots, z_n)=\frac{1}{\vol( K)} \int_{K}e^{2i\pi(\sum_{k=1}^n z_i x_i)} dx_1\dots dx_n.
$$
For $\theta \in S^{n-1}$ and $z\in {\mathbb C}$,  let $F_\theta(z)=F(z\theta)$.  Then, by the easy part of Paley-Wiener theorem, $F_\theta $ is  an even entire function of exponential type $[-\|\theta\|_{K^*}^{-1},  \|\theta\|_{K^*}^{-1}]$. Moreover, since $F$ is entire and even, there exists  an entire function $H_\theta:{\mathbb C}\to {\mathbb C}$ such that $H_\theta(z^2)= F_\theta(z)=F(z\theta)$. Finally, we define $R_\theta: {\mathbb C}^n  \to{\mathbb C}$ as a radial extension of $F_\theta$ by
$$R_\theta(z_1,\dots, z_n)=H_\theta(z_1^2+\dots +z_n^2).$$ 
Note that $z\mapsto R_\theta(z)$ is an entire function which satisfies $R_\theta(0)=F(0)=1$.  Moreover, since  $F_\theta$ is of exponential type $[-\|\theta\|_{K^*}^{-1}, \|\theta\|_{K^*}^{-1}]$,  $R_\theta $ is of exponential type $\|\theta\|_{K^*}^{-1}B_2^n$. It follows, from the Paley-Wiener theorem, that  the support of  the restriction to $\R^n$ of $\widehat{R_\theta}$ is contained in  $(\|\theta\|_{K^*}^{-1}B_2^n)^*=\|\theta\|_{K^*}B_2^n.$  Since $R_\theta\in L_2(\R^n)$, one may write, using the Plancherel equality and the Cauchy-Schwarz inequality with $v_n=\vol(B_2^n)$, 
\begin{align} \label{CSfourier}
    \int_{\R^n} |R_{\theta}(x) |^2dx=&\int_{\|\theta\|_{K^*} B_2^n} |\widehat{R_{\theta}}(x) |^2dx\ge \frac{|\int_{\|\theta\|_{K^*} B_2^n} \widehat{R_{\theta}}(x) dx|^2}{v_n \|\theta\|_{K^*}^n}\\
    =&\frac{|\int_{\R^n} \widehat{R_{\theta}}(x) dx|^2}{v_n \|\theta\|_{K^*}^n}
=\frac{R_{\theta}(0)}{ v_n\|\theta\|_{K^*}^n }= \frac{1}{v_n\|\theta\|_{K^*}^n }. \nonumber
\end{align}
If $|x|=r=|re_1|$, one has 
\begin{equation}\label{rth}
R_\theta(x)=F(|x|\theta)=F(r\theta)=R_\theta(re_1),
\end{equation}
so that $R_\theta$ is rotation invariant on $\R^n$.
Thus 
\begin{align*}
 \frac{1}{\vol(K) }=& \int_{\R^n} |{\widehat F}(x)|^2 dx=
  \int_{\R^n} | F(x)|^2 dx=
 \int_{S^{n-1}}\int_0^{+\infty} |F(r\theta)|^2 r^{n-1}dr d\theta\\
 =&\int_{S^{n-1}}\int_0^{+\infty} |R_\theta(re_1)|^2 r^{n-1}dr d\theta
 =\frac{1}{nv_n} \int_{\R^n} |R_\theta(x)|^2dx.
 \end{align*}
 It follows that
$$\frac{1}{\vol(K) }
= \frac{1}{nv_n} \int_{S^{n-1}} \int_{\R^n} |R_\theta(x)|^2 dxd\theta\ge \frac{1}{  nv_n^2}  \int_{S^{n-1}} \|\theta\|_{K^*}^{-n} d\theta=\frac{\vol(K^*)}{\vol(B_2^n)^2}, $$
which is the Blaschke-Santal\'o inequality.

\vskip 1mm\noindent
Bianchi and Kelly \cite{BK} also provided a proof of the equality case. Indeed, 
assume that we have equality in the  Blaschke-Santal\'o inequality, then we must have equality  in (\ref{CSfourier}). Thus, for every $\theta\in S^{n-1}$ and some $c_{\theta} \in \R$, one has $\widehat{R_{\theta}}=c_{\theta} {\bf 1}_{\|\theta\|_{K^*} B_2^n}$ on $\R^n$ and 
$$R_{\theta}(x)= \int_{\R^n} e^{-2i\pi \langle x, y \rangle} \widehat{R_{\theta}}(y) dy=c_\theta \int_{\|\theta\|_{K^*} B_2^n} e^{-2i\pi \langle x, y \rangle} dy.$$
Moreover, since $R_{\theta}(0)= 1$,  one gets $c_{\theta} \vol(\|\theta\|_{K^*} B_2^n)=1$. Next, by (\ref{rth}), we get 
\begin{equation}\label{fur}\frac{1}{\vol( K)} \int_{K}e^{-2i\pi r\langle \theta, y \rangle}  dy=F(r\theta)=\frac{1}{\vol(\|\theta\|_{K^*} B_2^n)} \int_{\|\theta\|_{K^*} B_2^n} e^{-2i\pi r \langle \theta, y \rangle} dy.
\end{equation}
If $M$ is a convex body, $\theta\in S^{n-1}$ and $t\in \R$,  let $A_{M,\theta}(t)=\vol_{n-1}\big(M \cap (\theta^\perp +t \theta)\big)$. One has
$$
 \widehat{A_{K, \theta}}(r)=\int_\R  e^{-2i\pi r t }  A_{K, \theta}(t) dt = \int_{K}e^{-2i\pi r\langle \theta, y \rangle}  dy.
$$
Inverting the Fourier transform,  it follows from (\ref{fur}) that for all $t\in \R$ and $\theta \in S^{n-1},$ 
\begin{equation}\label{koldob}
\frac{1}{\vol( K)} A_{K, \theta}(t) =  \frac{A_{\|\theta\|_{K^*} B_2^n, \theta}(t)}{\vol(\|\theta\|_{K^*} B_2^n)} = \frac{A_{B_2^n, \theta}\left(t\|\theta\|_{K^*}^{-1}\right)}{\|\theta\|_{K^*}\vol( B_2^n)}.
\end{equation}
Now,  for $\theta \in S^{n-1},$ one has by (\ref{koldob})
\begin{align*}
\int_K \langle x,\theta\rangle^2dx=\int_{\R} t^2 A_{K, \theta}(t) dt 
&= \frac{\vol( K)}{\|\theta\|_{K^*}\vol( B_2^n)}\int_{\R} t^2A_{B_2^n, \theta}\left(t\|\theta\|_{K^*}^{-1}\right)dt\\
&=  \frac{\vol( K)}{\vol( B_2^n)}\|\theta\|_{K^*}^2 \int_{\R} u^2A_{B_2^n,\theta}(u) du
\end{align*}
and since by rotation invariance $A_{B_2^n,\theta}(u) $ does not depend on $\theta\in S^{n-1}$, one gets for some $c>0$ and all $\theta
\in S^{n-1}$,
$$\|\theta\|_{K^*}=c\left(\int_K \langle x,\theta\rangle^2dx\right)^{1/2}, $$ which proves that $K^*$ and thus $K$ is an ellipsoid (the last arguments
are inspired from \cite{MR0}).

\subsection{Further results and generalizations} Let us present a few  results which may be considered as offspring's of the Blaschke-Santal\'o inequality.
\subsubsection{Stability} K. B\"or\"ozcky \cite{Bor} established a stability version of the Blaschke-Santal\'o inequality, later improved by K. Ball and K. B\"or\"ozcky in \cite{BB}. Let  $d_{BM}(K,L)$ be the Banach-Mazur distance between two convex bodies $K$ and $L$ in $\R^n$:
$$
d_{BM}(K,L)=\inf\{ d>0: K-x\subseteq T(L-y)\subseteq d (K-x), \mbox{ for some  } T \in GL(n) \mbox{ and }x,y \in \R^n \}.
$$
The following stability theorem was proved in \cite{BB}:
\begin{thm} If $K$ is a convex body in $\R^n,$ $n\ge 3$, such that for some $\varepsilon>0$ one has
$$
(1+\varepsilon)\P(K)\ge \P(B_2^n),
$$
then 
$$
\log\big(d_{BM}(K, B_2^n)\big) \le c_n\varepsilon^{\frac{1}{3(n+1)}}|\log \varepsilon|^\frac{2}{3(n+1)}.
$$
where $c_n$ is an absolute constant depending on $n$ only.
\end{thm}
If in the above theorem we assume that $K$ is symmetric, then the exponent of $\varepsilon$ can be improved to $2/3(n+1).$

\subsubsection{Local and restricted maxima} After having proved that convex bodies with maximal volume product are ellipsoids, one may ask about the local maxima of  the volume product, in the sense of Hausdorff distance. Using Theorem \ref{shadow}, it was proved in  \cite{MR4} that any local maximum is an ellipsoid, which gives another proof of Blaschke-Santal\'o's inequality.
 
 One may also investigate maxima among certain classes of bodies not containing ellipsoids. For instance, in $\R^2$, among polygons with more than $m\ge 4$ vertices, the maxima are the affine images of regular polygons with $m$ vertices \cite{MR3}. For $n\ge 3$,  the  much more complicated situation was investigated by \cite{AFZ} using shadow systems. In particular, it was proved  in \cite{AFZ} that a polytope with maximal volume product among polytopes with at most $m$ vertices is simplicial (all its facets are simplices) and has exactly $m$ vertices. It was also proved that, among polytopes with at most $n+2$ vertices, the volume product is maximized by $\conv(\Delta_{\lceil{\frac{n}{2}}\rceil},\Delta_{\lfloor{\frac{n}{2}}\rfloor})$, where $\Delta_{\lceil{\frac{n}{2}}\rceil}$ and $\Delta_{\lfloor{\frac{n}{2}}\rfloor}$ are simplices living in complementary subspaces of dimensions $\lceil{\frac{n}{2}}\rceil$ and $\lfloor{\frac{n}{2}}\rfloor$  respectively  (by definition, for $\alpha\not\in\Z$, $\lfloor{\alpha}\rfloor$ is the integer part of $\alpha$ and  $\lceil{\alpha}\rceil=\lfloor{\alpha}\rfloor+1$, for
 $\alpha\in\Z$, $\lceil{\alpha}\rceil=\lfloor{\alpha}\rfloor=\alpha$). It is conjectured in \cite{AFZ} that, for $1\le k\le n$, among polytopes with at most $n+k$
vertices, the convex hull of $k$ simplices living in complementary subspaces of dimensions $\lceil{\frac{n}{k}}\rceil$ or $\lfloor{\frac{n}{k}}\rfloor$ have maximal volume product.
 
 Among unit balls of finite dimensional Lipschitz-free spaces, which are polytopes with at most $(n+1)^2$ extreme points, some preliminary results were established in \cite{AFGZ} and it was shown that the maximizers of the volume product are simplicial  polytopes.

\subsubsection{ $L_p$-centroid inequalities} In a series of works by Lutwak, Yang and Zhang \cite{LuZ, LuYZ, LuYZ2}, Blaschke-Santal\'o's inequality appears as a special case of a family of isoperimetric inequalities involving the so called $L_p$-centroid bodies and $L_p$-projection bodies. More precisely, consider a compact star-shaped body $K$ in $\R^n$ and $p\in [1,\infty],$ the polar $L_p$-centroid body  $\Gamma^*_pK$ is defined  via it's norm:
  $$
  \|x\|_{\Gamma_p^*K}^p=\frac{1}{c_{n,p}\vol(K)}
  \int_{K}|\langle x, y \rangle|^p dy,
  $$
here the normalization constant $c_{n,p}$ is chosen so that $\Gamma^*_pB_2^n=B_2^n.$  It was proved in \cite{LuZ} that for all $p\in [1,\infty]$
\begin{equation}\label{LZ}
\vol(K)\vol(\Gamma_p^*K)\le \vol(B_2^n)^2,
\end{equation}
with equality if and only if $K$ is an ellipsoid centered at the origin. It turns out that if $K$ is a centrally symmetric convex body then $\Gamma_\infty^*K=K^*$ and thus the symmetric case of the Blaschke-Santal\'o inequality follows from (\ref{LZ}) when $p=\infty$. A stronger version of  (\ref{LZ}) was proved in \cite{LuYZ}: 
$$
\vol(\Gamma_pK)\ge \vol(K),
$$
for any star body  in $\R^n$ and $p\in [1,\infty]$. This inequality, for $p=1$, links the theory to the Busemann-Petty centroid inequality  \cite{Pe1} see also \cite{Ga,Sc}.
 If $K$ and $L$ are compact subsets of  $\R^n$, then for $p \ge 1$, it was proved in \cite[Corollary 6.3]{LuYZ2}  that for some $ c(p,n)>0$, one has
$$\int_{K\times L}
|\langle x , y\rangle |^p dx dy \ge  c(p,n)\big(\vol(K)\vol(L)\big)^{\frac{n+p}{n}}$$
with equality if and only if $K$ and $L$ are, up to sets of measure 0, dilates of polar-reciprocal, origin-centered ellipsoids.
When $p\to +\infty$, one gets 
the following  version of the symmetric Blaschke-Santal\'o's in \cite{LuYZ2}:
  If $K,L$  are compact subsets of $\R^n$, then
$$
\vol(B_2^n)^2 \max\limits_{x\in K,y \in L}
|\langle x, y\rangle|^n \ge \vol (K)\vol(L).
$$

\subsubsection{Connection to affine quermassintegrals} Affine quermassintegrals were defined by Lutwak \cite{Lu0}. For $1\le k\le n$, the $k$-th affine quermassintegral of a convex body $K$ is:
$$
\Phi_k(K)=\frac{v_n}{v_k}\left(\int_{Gr(k,n)}\vol_k(P_FK)^{-n}\sigma_{n,k}(dF)\right)^{-1/n},
$$
where $Gr(k,n)$ is the Grassmann manifold of $k$-dimensional linear subspaces $F$ of $\R^n$, $\sigma_{n,k}$ is Haar probability measure on $Gr(k,n)$ and $P_F$ is the orthogonal projection onto $F$. It was proved by Grinberg \cite{Gri} that $\Phi_k(K)$ is invariant under volume preserving affine transformations. Let $R_K>0$ satisfy $\vol(R_KB_2^n)=\vol(K)$. Lutwak  \cite{LuH} conjectured that for any convex body $K$ in $\R^n$ and any $k=1,\dots, n-1$, one has
\begin{equation}\label{eMY}
\Phi_k(K) \ge \Phi_k(R_KB_2^n)
\end{equation}
with equality if and only if $K$ is an ellipsoid. This conjecture was open for quite a long time. Lutwak proved that, for $k=1$, it follows directly  from the Blaschke-Santal\'o inequality (and that the case $k=n-1$ is connected to an inequality of Petty \cite{Pe2} \cite{Pe3}). Recently, E. Milman and Yehudayoff  \cite{MY} proved that this conjecture is true.  As one of the steps in the proof, they showed that $\Phi_k(K) \ge \Phi_k(S_HK)$, generalizing the previous result of \cite{MP1}.  In addition, a simplified proof of the Petty projection inequality was presented in \cite{MY}. Those interesting results suggest that (\ref{eMY}) can be viewed as a generalization of the Blaschke-Santal\'o inequality.

\subsubsection{A conjecture of K. Ball} Keith Ball \cite{Ba0} conjectured that if  $K$ is a convex symmetric body in $\R^n$ then
\begin{equation}\label{conjB}
\int_K\int_{K^*}\langle x,y\rangle^2dxdy \le  \int_{B_2^n}\int_{B_2^n}\langle x,y\rangle^2dxdy=\frac{n}{(n+2)^2} \vol(B_2^n)^2.
\end{equation}
and he proved a kind of reverse inequality:
$$
\frac{n \big(\vol(K)\vol(K^*)\big)^{\frac{n+2}{n}}}{(n+2)^2 \vol(B_2^n)^{\frac{4}{n}}} \le \int_K\int_{K^*}\langle x,y\rangle^2dxdy,
$$
which shows that  inequality (\ref{conjB}) is stronger than the Blaschke-Santal\'o inequality.  In \cite{Ba0, Ba1}, (\ref{conjB}) was proved for unconditional bodies. Generalizations are considered in \cite{KaSa} and \cite{Fa} (see section~\ref{Funk} for the later).

\subsubsection{Stochastic and log-concave measures extensions}\label{section:stoch}

Following the ideas initiated in \cite{PP}, the authors of \cite{CFPP} pursued a probabilistic approach of the Blaschke-Santal\'o inequality for symmetric bodies and established the following result.

\begin{thm}\label{cfpp}
For $N,n \ge 1$, let $(\Omega,\mathcal{B}, P)$ be a probability space and
\begin{itemize}
\item $X_1,\ldots, X_N:\Omega\to \R^n$ be independent random vectors, whose laws have densities with respect to Lebesgue measure which are bounded by one.
\item $Z_1,\dots, Z_N: \Omega\to \R^n$ be  independent random vectors uniformly distributed in $rB_2^n$ with $\vol(rB_2^n)=1$. 
\item  $\mu$ be the rotation invariant  measure on $\R^n$ with density $e^{\varphi(|x|)}$, $x\in \R^n$ with respect to Lebesgue measure, where $\varphi:\R_+\to\R_+$ is a non-increasing function.
\item $C_{X,N}(\omega)=\conv(\pm X_1(\omega),\ldots,\pm X_N(\omega))$ and  $C_{Z,N}(\omega)=\conv(\pm Z_1(\omega),\ldots,\pm Z_N(\omega))$ for $\omega\in \Omega$.
\end{itemize}
Then  for all $t\ge 0$, one has
$P(\{\omega\in \Omega; \mu(C_{X,N}(\omega)^*)\ge t\})\le
P(\{\omega\in \Omega; \mu(C_{Z,N}(\omega)^*)\ge t\})$.
\end{thm}

It follows of course that the same comparison holds in expectation.  The tools used there are shadow systems as in the work of Campi and
Gronchi~\cite{CG1}, together with the rearrangement inequalities of Rogers \cite{R} and Brascamp-Lieb-Luttinger \cite{BLL}.
Applying Theorem \ref{cfpp} to $X_1, \dots, X_N$  uniformly distributed  on  a convex body $K$ and using  that when  $N\to +\infty$, the sequence of random polytopes $
P_{K,N}:=\conv(\pm X_1,\ldots, \pm X_N)$ 
converges almost surely to $K$ in the Hausdorff metric, we deduce that for  measures $\mu$, as in theorem \ref{cfpp}, one has 
\[
\mu(K^*)\le\mu((R_KB_2^n)^*)=\mu\left(\frac{B_2^n}{R_K}\right), \quad\hbox{where $R_K=\left(\frac{\vol(K)}{\vol(B_2^n)}\right)^\frac{1}{n}$}.
\]
Since clearly $\mu(K)\le\mu(R_K B_2^n)$, we deduce that $\mu(K)\mu(K^*)\le \mu(R_K B_2^n)\mu(B_2^n/R_K)$.
If, moreover, $t\mapsto \varphi(e^t)$ is concave, then $t\mapsto\mu(e^tB_2^n)$ is also log-concave (see \cite{CFM}). Thus, it follows that for such measures $\mu$ and for any symmetric convex body $K$, one has
\begin{equation}\label{bs-for-mu}
\mu(K)\mu(K^*)\le\mu(B_2^n)^2.
\end{equation}
It was proved in \cite{CR} that under those hypotheses,  $t\mapsto \mu(e^tK)$ is log-concave (extending the same property for Gaussian measures established in \cite{CFM}). 
It  was asked in \cite{Co} whether \eqref{bs-for-mu}  holds for all symmetric log-concave measures $\mu$. 

We shall prove \eqref{bs-for-mu}  when moreover
$\mu$ has an unconditional density $f$ with respect to the Lebesgue measure (a function $f:\R^n\to\R$
 is  said {\it unconditional} if for some basis $e_1,\dots, e_n$ of $\R^n$, one has for all $(\varepsilon_1,\dots, \varepsilon_n)\in\{-1;1\}^n$ and $(x_1,\dots,x_n)\in\R^n$, $f(\sum_{i=1}^n x_ie_i)=f(\sum_{i=1}^n \varepsilon_ix_ie_i)$).
\begin{thm}
If $\mu$ a measure on $\R^n$ with  an unconditional and log-concave density with respect to the Lebesgue measure and $K$ is a symmetric convex body in $\R^n$, then  $\mu(K)\mu(K^*)\le\mu(B_2^n)^2.$
\end{thm}

\begin{proof}  We apply  first a linear transform making the density of $\mu$ unconditional with respect to the canonical basis of $\R^n$. 
Let $H$ be a coordinate hyperplane and let $S_HK$ be the Steiner symmetral of $K$ with respect to $H$. Using \eqref{steiner-inclusion} as in the proof of Meyer-Pajor \cite{MP1}  (see section \ref{steiner-mp} above), we get  $\mu(K^*)\le\mu((S_HK)^*)$. Moreover, it is easy to see that $\mu(K)\le\mu(S_HK)$. Thus, denoting by $L$ the convex body obtained from $K$ after $n$ successive Steiner symmetrisation with respect to the coordinate hyperplanes, we get $\mu(K)\mu(K^*)\le\mu(L)\mu(L^*)$. We are now reduced to the case when $\mu$ and $K$ are unconditional.  Using the classical Prékopa-Leindler inequality (see for example \cite[page 3]{Pi}), it was shown in \cite{FM1} that then $\mu(L)\mu(L^*)\le\mu(B_2^n)^2$.
\end{proof}
\subsubsection{Blaschke-Santal\'o type inequality on the sphere}

Another inequality of Blaschke-Santal\'o type was  established by Gao, Hug and Schneider \cite{GHS} on the sphere. We define the polar of  $A\subset S^{n-1}$ by
\[
A^\circ:=\{y\in S^{n-1}; \langle x,y\rangle\le 0, \mbox{ for all } x\in A\}.
\]
If $\pos(A):=\{t x; x\in A,\ t\ge0\}$, then $A^\circ=(\pos(A))^*\cap S^{n-1}$. Let $\sigma$ be the Haar probability measure on $S^{n-1}$. A {\it spherical cap} is the non-empty intersection of $S^{n-1}$ with a halfspace.   This work was further  generalized by Hu and Li \cite{HuLi} who proved a number of Blaschke-Santal\'o type
inequalities in the sphere and hyperbolic space.

\begin{thm}\cite{GHS}
Let $A$ be a non-empty measurable subset of $S^{n-1}$ and $C$ be a spherical cap such that $\sigma(A)=\sigma(C)$. Then $\sigma(A^\circ)\le\sigma(C^\circ)$. If moreover $A$ is closed and $\sigma(A)<1/2$, there is equality if and only if $A$ is a spherical cap. 
\end{thm}

Two proofs were given in \cite{GHS}. One of them  uses a special type of symmetrization called the two-point symmetrization and  for the equality case the results of \cite{AF}. Hack and Pivovarov  \cite{HP} gave a stochastic extension of  theorem 7 in the spirit of Theorem \ref{cfpp}.

\section{Mahler conjecture. Special cases}\label{MC}

The problem of the lower bound of $\VP(K)$ is not yet solved, although significant progresses were done these last years.
The first results are due to Mahler for  $n=2$, who proved that $\VP(K)\ge \VP(\Delta_2)=\frac{27}{4}$ where $\Delta_2$ is a triangle and in the centrally symmetric case that $\VP(K)\ge \VP([-1,1]^2)=\frac{8}{3}$ (see also \cite{To}). For the proofs, he used polygons and could not thus give the case of equality.  Observe that he continued to be interested in this problem \cite{Ma3,Ma4}.
 The case of equality in  dimension $2$ was obtained by Meyer \cite{Me2} for general bodies and by Reisner  \cite{Re1} (see also \cite {SR1, Me1, To}) for centrally symmetric bodies. What happens in dimension $n\ge 3$? There are  two conjectures, the first one formulated explicitly by Mahler \cite{Ma1}, but not the second one.

\begin{conj}\label{mahler}   For every convex body  $K$ in  $\R^n$, one has 
 $$\VP(K)\ge \VP(\Delta_n)=\frac{(n+1)^{n+1}}{(n!)^2},$$ where $\Delta_n$ is a simplex in $\R^n$, with equality if and only if $K=\Delta_n$.
\end{conj}
\begin{conj}\label{conjcube}  For every centrally symmetric convex body $K$ in dimension $n$, one has $$\VP(K)\ge \VP(B_\infty^n)=\frac{4^n}{n!},$$ 
 where $B_\infty^n= [-1,1]^n$ is a cube, with equality if and only if $K$ is a Hanner polytope (see Definition \ref{hanner} below).
 \end{conj}

\subsection{The conjectured  minimum in the symmetric case is not unique}  To understand conjecture \ref{conjcube} and different phenomena related to it, we define  Hanner polytopes \cite{Ha}, and first the $\ell_1$-sum $E\oplus_1 F$ and $\ell_{\infty}$-sum $E\oplus_{\infty} F$ of two normed spaces $E$ and $F$.

\begin{deff}
Let $(E, \|\cdot\|_E)$ and $(F,\|\cdot\|_F)$ be two normed spaces. Then on $E\times F$, we define two norms:
 the norm of  the $\ell_{\infty}$-sum $E\oplus_{\infty} F$ of $E$ and $F$ and of their $\ell_1$-sum  $E\oplus_1 F$ by
\begin{itemize}
\item $\| (x,y) \|_{\infty}=\max(\| x \|_E, \| y \|_F)$.
\item $ \|(x,y) \|_{1}= \| x \|_E + \| y \|_F$.
 \end{itemize}
\end{deff}
We note that if $E$ and $F$  are normed spaces then the unit ball of their $\ell_{\infty}$-sum is the Minkowki sum of the unit balls of $E$ and $F$ in $E\times F$ and the unit ball of their $\ell_{1}$-sum is their convex hull.  Analogously, if we consider two convex bodies $K\subset \R^{n_1}$ and $L\subset \R^{n_2}$, we  define two convex bodies in $\R^{n_1+n_2}$:
\begin{itemize}
\item $K\oplus_\infty L=K\times\{0\}+\{0\}\times L=\{x_1+x_2: x_1 \in K, x_2\in L\}$, their $\ell_\infty$-sum.
\item  $K \oplus_1 L=\conv(K\times\{0\},\{0\}\times L)$, their $\ell_1$-sum.
\end{itemize}
One major property of $\ell_{1}$ and $\ell_{\infty}$-sums is that
\begin{equation}\label{oneinf}
(K\oplus_\infty L)^*=K^*\oplus_1 L^*.
\end{equation}

Now we are ready to define Hanner polytopes. 

\begin{deff}\label{hanner}  In dimension $1$, Hanner polytopes are  symmetric segments. Suppose that Hanner polytopes are defined in all dimension $m\le n-1$. A   Hanner 
polytope in dimension $n$ is the unit ball of an $n$-dimensional normed space $H$ such that for some  $k$-dimensional  subspace $E$,  
$1\le k\le n$, and 
$(n-k)$-dimensional subspace $F$ of $H$, whose unit balls are Hanner polytopes, $1\le k\le n-1$,  $H$ is the $\ell_{\infty}$-sum or the $\ell_1$-sum of $E$ and $F$. 
\end{deff}
Let us now discuss the basic properties of Hanner polytopes: 
\begin{itemize}
\item In $\R^2$, there is a unique (up to isomorphism) Hanner polytope, which is the square. 
\item In $\R^3$, there
 are exactly $2$ (up to isomorphism) Hanner polytopes, which are the cube
and the centrally symmetric octahedron.
\item In  $\R^4$, there are, up two isomorphism, $4$ different classes of Hanner polytopes, including two which are not isomorphic to the cube or  the crosspolytope. And in $\R^n$,  their number increases quickly with $n$.
\item 
The normed spaces  whose unit balls $K$ are Hanner polytopes are up to isometry exactly those which satisfy the $3-2$-intersection property: for any three vectors $u_1, u_2$ and $u_3$ if  $(K+u_i)\cap (K+u_j) \not=\emptyset,$ for all $1 \le i<j \le 3$, then  the intersection of all  $3$ balls is not empty \cite{HL}.
\item  A  Hanner polytope is unconditional (see Definition \ref{uncond} below). 
\item  If $K$ is a Hanner polytope, then  so is $K^*$. This follows from (\ref{oneinf}).

\item If  $K\subset \R^{n_1}$ and $L\subset \R^{n_2}$   are two convex bodies, then
$$\VP(K\oplus_{\infty} L)= \VP(K\oplus_1 L)=\frac{n_1!n_2!}{ (n_1+n_2)! }\VP(K)\VP(L).$$

\item Using induction, it follows  that the volume product of a Hanner polytope in $\R^n$ is $\frac{4^n}{n!}$.
 \end{itemize}
 
  In some sense, Conjecture 1 seems easier than Conjecture 2 because up to an isomorphism, there is only one proposed minimum. But polarity is done with respect to the Santal\'o point of a convex body $K$, which is not always well located, so that one has to prove that for every $z\in \inte(K)$, $\vol(K)\vol(K^z)\ge \VP(\Delta_n)$.  Observe  however that if $K$  has minimal volume product among all other convex bodies, then its Santal\'o point is also its center of gravity.

 \subsection{The planar case} First, note that the conjecture holds  with the case of equality for $n=2$ (Mahler \cite{Ma1}, Meyer\cite{Me2} for another proof and the case of equality). Let us sketch a proof of the planar case and use this opportunity to give an example of how the method of shadow systems as well as Theorem \ref{shadow} can be used;  note that the method in this case can be traced back to the original proof from \cite{Ma1} and is almost identical for the general and the symmetric case.  We concentrate on the general case.
 
 \begin{proof}(Lower bound in $\R^2$) It is enough to show that $\VP(T)\ge \VP(\Delta_2)$ for all    convex polygons $T \subset \R^2$.
  The main idea  is  to remove vertices of $T$.  We use induction on the number  $k$ of vertices.   Let $T$ be a polygon with $k\ge 4$ vertices. Suppose that
$T=\conv(v_1,v_2,v_3,\dots,v_k)$, with $v_1,v_2,v_3,...,v_k,$ written in the clockwise order. 
We shall prove that $\VP(T)\ge \VP(Q)$, for a polygon $Q$ with only $k-1$ vertices.  For $i\not= j$, let $\ell_{i,j}$ be a line through $v_i$ and $v_j$.  Let $\theta \in S^1$ be  parallel to  the  line $\ell_{1,k-1}$. And define $T_t=\conv(v_1,v_2,\dots,v_{k-1},v_k+t\theta)$ (i.e. we move $v_k$ on a line parallel to $\ell_{1,k-1}$).
The line $\{v_k+t\theta; t\in \R\}$ meets $\ell_{k-1,k}$ at $v'_k$  when $t=a$ and $\ell_{1,2}$ at  $v'_1$ when $t=b$.  Since $T_0=T$, one may assume that $a<0<b$. It is easy to see that, for $t\in [a,b]$, $t\mapsto T_t$ is a shadow system with $\vol(T_t)=\vol(T)$. By Theorem \ref{shadow},  $t\mapsto \VP(T_t)^{-1}$ is convex on the interval $[a,b]$ and thus is maximal at its end points. Thus $\VP(T)\ge \min(\VP(T_a),\VP(T_b))$ where
$T_a=\conv(v_1,  \dots,v_{k-2}, v'_k)$ and $T_b= \conv(v'_1,v_2,\dots, v_{k-1})$ are polygons with only $k-1$ vertices.
\end{proof}

\begin{remark} The above method was used to prove a number of partial cases of Mahler's conjectures (see \cite{MR2, FMZ,AFZ, AFGZ, Sar}). Unfortunately, there seems to be no way to generalize this approach to dimension $3$ and higher, one of the reason is that if a vertex $v$ of  a polytope $P$ may be a vertex of  a lot of non simplicial faces, and how  "moving" $v$ without breaking the combinatorial structure of $P$? And when the combinatorial structure of $P$ is broken, it is difficult to compute  volumes.
\end{remark}

\begin{remark} 
In \cite{Reb}, Rebollo Bueno established also stochastic versions of the planar case of Mahler's conjectures. With the notations of section \ref{section:stoch}, he proved that for any centrally symmetric convex body $K$ in the plane and any $r\ge1$, 
\[
\E(\vol(P_{K,N}^*)^{-r})\le \E(\vol(P_{Q,N}^*)^{-r}),
\]
where $Q$ is a square with $\vol(Q)=\vol(K)$. For $r=1$ and $N\to+\infty$, this gives back the planar case of Mahler's conjecture. The same type of result is also established in \cite{Reb} for general convex bodies in the plane. 
\end{remark}

\subsection{The case of zonoids} The conjecture holds for zonoids and polar of zonoids, with  equality case  for cubes  (Reisner \cite{Re1, Re2}
   and Gordon, Meyer and Reisner  \cite{GMR} for a second proof). We recall that a {\it zonoid}  in $\R^n$ is a  Hausdorff limit of {\it zonotopes}, that is of finite sums of segments. Since a segment is symmetric with respect to its midpoint, any zonotope, and thus any zonoid is centrally symmetric. From now, when speaking of a zonoid $Z$, we shall suppose  that $Z=-Z$.
Also, the polar bodies of zonoids can be seen as the unit balls of finite dimensional subspaces of $L_1([0,1], dx)$. Observe  that every convex centrally symmetric body in $\R^2$ is a zonoid.
We refer to \cite{Bo, GW, Sc}    
  for basic properties of zonoids. 

  \begin{proof}(The lower bound of volume product for zonoids  \cite{GMR})
  For a  zonoid $Z\subset \R^n$ , there exists a measure $\mu$ on $S^{n-1}$ such that $h_Z(x)=\frac{1}{2}\int_{S^{n-1}} |\langle x, u\rangle| d\mu(u)$ for all $x\in \R^n$. 
 Since $\vol(Z)=\frac{1}{n} \int_{S^{n-1}} \vol_{n-1}(P_{u^\perp}Z) d\mu(u)$, one has
 \begin{align*}
 \vol(Z^*)\int_{S^{n-1}}\vol_{n-1}(P_{u^\perp}Z) d\mu(u) &=n\vol(Z)\vol(Z^*)=
 \frac{n+1}{2} \vol(Z)\int_{Z^*} h_K(x) dx\\
 &=
\frac{n+1}{2} \vol(Z)\int_{S^{n-1}} \left(\int_{Z^*} |\langle x, u\rangle| dx\right) d\mu(u).
\end{align*}\
 It follows that for some $u\in {S^{n-1}}$, one has
 $$\vol(Z^*)\vol_{n-1}(P_{u^\perp}Z) \le  \frac{n+1}{2}  \vol(Z) \int_{Z^*} | \langle x, u\rangle | dx.$$
 Now $\int_{Z^*} | \langle x, u\rangle | dx= 2\int_0^{\infty} tf(t) dt$, where $f(t)=\vol\big(Z^*\cap(u^{\perp} +tu)\big)$ is the volume in $u^{\perp} $ of the sections of 
 $Z^*$ with hyperplanes parallel to $u^{\perp}$.  Note that $f(0)=\vol(Z^*\cap u^{\perp})$ and $2\int_0^{\infty} f(t) dt=\vol(Z^*)$. 
 By the Brunn-Minkowski theorem, the function $f^{\frac{1}{n-1}}$ is concave on its support. By a classical  estimate (see for instance \cite{MiP}),
$$
\int_0^{\infty} tf(t) dt \le \frac{n}{n+1} \frac{(\int_0^{\infty} f(t) dt)^2 }{f(0)},$$ 
with equality if and only if $f(t)=f(0) (1-ct)^{n-1}_+$, for some $c>0$ and all $t\ge 0$. 
This gives $$\int_{Z^*} | \langle x, u\rangle | dx\le  2\frac{n}{n+1} \frac{4^{-1 } \vol(Z^*)^2}{  \vol_{n-1}(Z^*\cap u^{\perp}) }=
\frac{n}{2(n+1) } 
\frac{\vol(Z^*)^2 }{  \vol_{n-1}(Z^*\cap u^{\perp} ) }, $$
and thus
  $$\vol(Z^*)\vol_{n-1}(P_{u^\perp}Z)\le  \frac{n+1}{2}  \vol(Z)  
\frac{n}{2(n+1) }  \frac{\vol(Z^*)^2}{  \vol_{n-1}(Z^*\cap u^{\perp}) }, $$
  so that $$\vol(Z) \vol(Z^*)\ge \frac{4}{n}\vol_{n-1}(P_{u^\perp}Z)\vol_{n-1}(Z^*\cap u^{\perp}),$$
 which allows to conclude by induction, with the case of equality, since $P_{u^\perp}Z$ is a zonoid in dimension $n-1$ and $(P_{u^\perp}Z)^*= Z^*\cap u^{\perp}$.
 \end{proof}

 \begin{remark} Campi and  Gronchi \cite{CG2} presented a very interesting inequality on the volume of $L_p$-zonotopes, which  givesinequality, in particular,  another proof of the above result. It is interesting to note that the proof in  \cite{CG2} is based on the shadow systems technique. Another proof using shadow systems was presented by Saroglou in  \cite{Sar}.
 \end{remark}

 \begin{remark} Marc Meckes \cite{Mec}gaveanother proof of  Mahler's conjecture for zonoids, based on the notion of {\it magnitude}  introduced by  Leinster \cite{Lei}, which is a numerical isometric invariant for metric spaces. He studies  the magnitude of a convex body in hypermetric normed spaces (which include $\ell_p^n$, $p\in [1,2])$ and proved  a new upper bound for magnitude on such spaces using the Holmes-Thompson intrinsic volumes of their unit balls.
 \end{remark}
   
\subsection{The case of unconditional bodies}
\begin{deff}\label{uncond} Let $K$ in $\R^n$ be  a convex body. We say that $K$  is {\em unconditional } if  for some basis $e_1, \dots, e_n$ of $\R^n$ one has $x_1e_1+\dots + x_n e_n\in K$ if and only if $|x_1|e_1+\dots + |x_n| e_n \in K$. We say that $K$ is  {\em almost  unconditional} if for some basis $e_1, \dots, e_n$ of $\R^n$ for every $1\le i\le n$, one has $P_i K=K\cap H_i$, where $H_i$ is  linear span of $ \{e_j, j\not=i\} $ and $P_i$ is the linear projection from $\R^n$ onto $H_i$ parallel to $e_i$.
\end{deff}
If $K$ is unconditional, after a linear transformation which does not change $\VP(K)$,
we  may suppose that  $(e_1, \dots, e_n)$ is the canonical basis of $\R^n$. Unconditional bodies are almost unconditional and centrally symmetric.  Observe also that 
if $K$ is unconditional (resp. almost unconditional) with respect to some basis, then $K^*$ is also
 unconditional (resp. almost unconditional)  with respect to the dual basis.

We follow the proof of \cite{Me1} of the inequality $\VP(K)\ge \VP(B_\infty^n)$ (the first proof was given in \cite{SR1}). We don't prove  the case of equality (Hanner polytopes), which is more involved.
\begin{proof}
We use induction on $n$. It is trivial for $n=1$. We suppose that  $e_1, \dots, e_n$ is the canonical basis of $\R^n$.
Let $K_+=K\cap \R_+^n$, ${K^*}_+= K^*\cap \R_+^n$. Then 
$\VP(K)=4^n \vol(K_+)\vol(K^*_+)$.
For $x\in \R^n_+$, one has
$$ x\in  K_+\hbox{ if and only if }\langle x,y\rangle \le 1\hbox{  for any }y\in K^*_+,$$
$$ y\in  K^*_+\hbox{ if and only if }\langle x,y\rangle \le 1\hbox{  for any }x\in   K_+.$$
For $1\le i \le n$, $K_i:= K\cap\{x_i=0\}$ is  an unconditional body in $\R^{n-1}$ and  $(K_i)^*= (K^*)_i$. Let $(K_i)_+ =K_i\cap (\R^+)^n$. For $x=(x_1,\dots, x_n)\in K_+$, let $C_i(x)$  be the convex hull of $\{x\}$ with $(K_i)_+$. Since $C_i(x)$ is a cone with apex $x$ and basis $(K_i)_+$, one has
$$\vol \big(C_i(x)\big)=
\frac {x_i}{n}\vol_{n-1}\big((K_i)_+\big).$$
Thus
\begin{equation}\label{meyer1}
\vol(K_+) \ge \vol \big(\cup_{i=1}^n C_i(x)\big)= \sum_{i=1}^n \vol \big(C_i(x)\big)=\frac{1}{ n} \sum_{i=1}^n x_i\vol_{n-1}\big( (K_i)_+\big).
\end{equation}
Let $a:=\frac{1}{n\vol(K_+) }\Big(\vol_{n-1}\big( (K_1)_+\big),\dots, \vol_{n-1}\big( (K_n)_+\big)\Big)$ in $\R^n$. By (\ref{meyer1})
one has $\langle a,x\rangle \le 1$ for all $x\in K_+$, that is
$a\in K^*_+$.
Also, $a^*:=\frac{1}{n\vol(K^*_+) }  \Big(\vol_{n-1} \big( (K^*_1)_+\big),\dots, \vol_{n-1} \big(    (K^*_n)_+  \big)\Big)\in K_+$. Thus $\langle a,a^*\rangle\le 1$, that is
$$\frac{\sum_{i=1}^n \vol_{n-1}\big((K_i)_+\big)\vol_{n-1}\big((K^*_i)_+\big)}{n^2 \vol(K_+)\vol((K^*_+)} \le 1, $$
so that 
$$\VP(K)=4^n \vol(K_+)\vol(K^*_+)\ge \frac{4^n}{ n^2} \sum_{i=1}^n \vol_{n-1}\big( (K_i)_+\big)\vol_{n-1}\big( (K^*_i)_+\big).$$
For $1\le i \le n$, one has $\vol_{n-1}(K_i)= 2^{n-1}\vol_{n-1}\big( (K_i)_+\big)$ and $\vol_{n-1}(K^*_i)= 2^{n-1}\vol_{n-1}\big( (K^*_i)_+\big)$. Since the $K_i$ are also unconditional, the induction hypothesis  gives $\VP(K_i)\ge \frac{ 4^{n-1}}{(n-1)!}$, $1\le i\le n$. Thus 
$$\VP(K)\ge \frac{4}{ n^2} \sum_{i=1}^n \vol_{n-1}(K_i)\vol_{n-1}(K^*_i)\ge \frac{4}{ n^2} \cdot n \cdot\frac{4^{n-1} }{(n-1)!}= \frac{4^n}{n!}.
$$
\end{proof}
\begin{remark} A small modification of this proof allows to treat the case of almost unconditional centrally symmetric bodies. Note that every centrally symmetric body in $\R^2$ is almost unconditional.
\end{remark}

\subsection{The 3-dimensional symmetric case}
The symmetric case in  $\R^3$ was solved by Irieh and Shibota \cite{IS1} in 2017 with a quite involved proof of about sixty pages. We would like here to highlight the main ideas  and to connect it with the unconditional case presented above. We will use the shorter proof  given in \cite{FHMRZ}.

A symmetric body $K \subset \R^n$, $n\ge 3$ is not generally almost  unconditional,  and thus not unconditional. However, every planar convex body has an almost unconditional basis.  For $n=3$, the goal is to show that a $3$-dimensional convex symmetric body $K$ may still have core properties of an unconditional body.
This is done with the help of the following equipartition result: 
\begin{thm}\label{thm:equipart}
Let $K \subset \R^3$ be a  symmetric convex body. Then there exist  3 planes $H_1,H_2,H_3$ passing through the origin such that:
\begin{itemize}
	\item they split $K$ into $8$ pieces of equal volume, and
	\item for each $i=1,2,3$, the section $K \cap H_i$ are split into $4$ parts of equal area by the other two planes.
\end{itemize}
\end{thm}
Note that theorem \ref{thm:equipart} belongs to the very rich theory of equipartitions. For example, a  celebrated result of Hadwiger \cite{Hadwiger}, answering a question of Gr\"unbaum \cite{Grunbaum}, shows that for any absolutely continuous finite measure in $\R^3$,  there exist three planes for which any octant has $1/8$ of the total mass. For proving Theorem \ref{thm:equipart},  one can use a result of Klartag (Theorem 2.1 of \cite{Kl1}); we refer to \cite{FHMRZ} for details. 

Our goal is to create an analog of formula (\ref{meyer1}). Consider  a sufficiently regular oriented hypersurface $A \subset \R^n$ and  define the vector 
$$
\overrightarrow{V}(A)=\int_A \overrightarrow{n_A}(x)dx,
$$
where $ \overrightarrow{n_A}(x)$ is the unit normal to $A$ at $x$ defined by its orientation. Next, for a convex body $K\subset \R^n$ with $0\in \inte(K)$, the orientation of a subset $A \subset \partial K$ is given by the outer normal $ \overrightarrow{n_K}$ to $K$.  If 
 $\C(A):=\{rx;\ 0\le r\le1, x\in A\}$, then
 $$
 \vol (\C(A))=\frac{1}{n}\int_A \langle x, \overrightarrow{n_K}(x)\rangle  dx.
$$
The following is a key proposition for our proof.
\begin{proposition}\label{Meyergen}
Let $K\subset \R^n$ be a convex body, with $0\in\inte(K)$, and let $A$ be a  Borel subset of $\partial K$ with $\vol(\C(A)) \not=0$, then  for all $x\in K$,
$$
\frac{1}{n} \langle x, \overrightarrow{V}(A)\rangle \leq \vol(\C(A)) {\rm{\ and\ thus \  }} \frac{\overrightarrow{V}(A)}{n\vol(\C(A))}\in K^*.
$$
\end{proposition}
\begin{proof} For all $x\in K$, we have $\langle x, \overrightarrow{n_K}(z)\rangle \le \langle z,  \overrightarrow{n_K}(z)\rangle$ for every $z\in\partial K$.  Thus for all $x\in K$, 
$$
\langle x, \overrightarrow{V}(A)\rangle = \int_A \langle x,  \overrightarrow{n_K}(z)\rangle  d z\leq \int_A \langle z, \overrightarrow{n_K}(z)\rangle  dz =n\vol(\C(A)). 
$$
\end{proof}

\begin{corollary}\label{corme}
Let $K$ be a convex body in $\R^n$ with $0\in\inte(K)$. If $A \subset \partial K$ and $B\subset \partial K^*$  are  Borel subsets such that $\vol(\C(A))>0$ and $\vol(\C(B))>0$, then 
 $$
 \langle \overrightarrow{V}(A), \overrightarrow{V}(B) \rangle  \le n^2\vol(\C(A))\vol(\C(B)).
 $$
\end{corollary}
\begin{proof} We  use the Proposition \ref{Meyergen} to get  $\frac{\overrightarrow{V}(A)}{n\vol(\C(A))}\in K^*$ and $\frac{\overrightarrow{V}(B)}{n\vol(\C(B))}\in K$.
\end{proof}
\noindent {\bf Proof of Conjecture 2 for $n=3$:} Since the volume product is continuous, it is enough to prove the conjecture for a centrally symmetric, smooth, strictly convex body $K$ (see \cite{Sc} Section 3.4).
From the linear invariance of the volume product, we may assume that the equipartition property obtained in  Theorem \ref{thm:equipart} is satisfied by the coordinates planes given by the canonical orthonormal basis $(e_1,e_2,e_3)$. 
As in  the unconditional case, we  divide $\R^3$ and the body $K$ into the octants defined by this basis, which define cones as in Corollary \ref{corme}. 
The main issue is that, in a sharp difference with the unconditional case, the dual cone to the cone defined as  an intersection of $K$ with an  octant is not the intersection of $K^*$ with this octant.  We will need  a bit of combinatorics to work around this issue.  

For $\eps\in\{-1;1\}^3$, let the $\eps$-octant be $\{x\in \R^3; \eps_i x_i\ge 0 \mbox{ for }   i=1,2,3\}$ and for $L\subset\R^3$, let $L_{\eps}$  be the intersection of $L$ with the $\eps$-octant: 
  $ 
L_{\eps}=\{x\in L; \eps_i x_i\ge0;\    i=1,2,3\}.
$   
 Let $N(\eps):=\{\eps'\in \{-1,1\}^3: \sum_{i=1}^3 |\eps_i-\eps'_i |=2\}$. Then $\eps' \in N(\eps)$ iff $[\eps, \eps']$ is an edge  $[-1,1]^3$.    

If  $K_{\eps} \cap K_{\eps'}$ is a hypersurface, we define $K_{\eps} \overrightarrow{\cap} K_{\eps'}$ to be oriented according to  the outer normals of $\partial K_{\eps}$.  Using Stokes theorem, we obtain 
$$
\overrightarrow{V}(\partial K_{\eps} )=\int_{\partial K_{\eps} }  \overrightarrow{n_{\partial K_{\eps}}}(x)dx
 -\sum_{\eps' \in N(\eps)} \overrightarrow{V}(K_{\eps} \overrightarrow{\cap} K_{\eps'}).
$$  
Using the equipartition of the areas of $K\cap  e_i^{\perp}$,
we get
$$
  \overrightarrow{V}(\partial K_{\eps})=-\sum_{\eps' \in N(\eps)} \overrightarrow{V}(K_{\eps} \overrightarrow{\cap} K_{\eps'} )=\sum_{i=1}^3 \frac{\vol(K\cap 
  e_i^{\perp})} {4}
  \eps_{i} \overrightarrow{e_i}.
$$

Let us look at the dual.
Since $K$ is strictly convex and smooth, there exists a diffeomorphism $\varphi:\partial K\to \partial K^*$ such that $\langle \varphi(x),x\rangle=1$ for all $x\in \partial K$.   
We extend $\varphi$ to $\R^3$ by homogeneity of degree one: $\varphi(\lambda x)=\lambda \varphi(x)$  for $\lambda\geq 0$.    
Then  
$$K^*=\bigcup_{\eps}\varphi (K_{\eps})\mbox{ and }\vol(K^*)=\sum_{\eps}\vol\big(\varphi (K_{\eps})\big).    $$
From the equipartition of volumes, one has
  $$
  \vol(K)\vol(K^*) =\sum_{\eps } \vol(K)\vol(\varphi (K_{\eps}))=8\sum_{\eps} \vol(K_{\eps})\vol\big(\varphi (K_{\eps})\big).
  $$  
From Corollary  \ref{corme}, we deduce that for $\eps\in\{-1,1\}^3$
\begin{eqnarray*}
\vol(K_{\eps})\vol\big(\varphi (K_{\eps})\big) \ge \frac{1}{9} \langle  \overrightarrow{V}(\partial K_{\eps}), \overrightarrow{V}\big(\varphi\big(\partial K_{\eps})\big)\rangle. 
\end{eqnarray*}  

Thus 
\begin{align}
\vol(K)\vol(K^*)&\ge\frac{8}{9}\sum_{\eps}  \langle \overrightarrow{V}( \partial K_{\eps}),\overrightarrow{V}\big( \varphi(\partial K_{\eps})\big)\rangle  \nonumber \\
&= 
\frac{8}{9} \sum_\eps\langle  \sum_{i=1}^3 \frac{ \vol(K\cap e_i^{\perp})}{4} \eps_i \overrightarrow{e_i},   \overrightarrow{V}( \varphi\big(\partial K_{\eps})\big)\rangle \nonumber\\
 &= \frac{8}{9}\sum_{i=1}^3 \frac{ \vol( K\cap {e_i}^{\perp})}{4} \langle \overrightarrow{e_i}, \sum_\eps \eps_i  \overrightarrow{V}( \varphi(\partial K_{\eps})\rangle.  \nonumber
\end{align}
Now we use Stokes theorem for $\varphi(\partial K)$ to get 
$$
\overrightarrow{V}\big(\varphi( (\partial K_{\eps})\big)=-\sum_{\eps' \in N(\eps)}\overrightarrow{V}\big(\varphi(K_{\eps}\overrightarrow{\cap} K_{\eps'}')\big).
$$
The next step requires a careful computation of the sums  following orientation of all surfaces, which  gives many cancellations. 
Next one combines the correct parts of $K$ and $\varphi(K)$ to get 
$$
\vol(K)\vol(K^*)\ge \frac{4}{9}\sum_{i=1}^3  \vol_{n-1}(K\cap e_i ^{\perp})\langle  \overrightarrow{e_i}, V \big(\varphi(K\cap  \overrightarrow{e_i}^\bot)\big)\rangle
$$
(see \cite{FHMRZ} for  the precise computations). Let $P_i$ be the orthogonal projection  onto ${e_i}^{\perp}$. 
Then $P_i: \varphi(K\cap  {e_i}^{\perp})\to P_i (K^*)$ is a bijection. Using Cauchy's formula for the volume of projections, we get
\begin{align*}\langle \overrightarrow{e_i}, V \big(\varphi(K\cap {e_i}^{\perp})\big)\rangle & =   \int\limits_{\varphi(K\cap {e_i}^{\perp})}  \langle \overrightarrow{n_{\varphi(K\cap  {e_i}^{\perp})}}(x),  \overrightarrow{e_i} \rangle dx \\ &=\vol_{n-1} \big(P_i(  \varphi(K\cap  {e_i}^{\perp}))\big) 
=   \vol_{n-1}\big(P_i(K^*)\big). 
\end{align*}
and if $\eps=(\eps_1,\dots, \eps_n)$, 
$$
\vol(K)\vol(K^*)\ge \frac{8}{9}\sum_{i=1}^3 \frac{ \vol_{n-1}\big(K\cap{e_i}^{\perp}\big)}{4} \langle \overrightarrow{e_i}, \sum_\eps \eps_i  \overrightarrow{V}\big( \varphi(\partial K_{\eps})\big)\rangle.  
$$ 
Finally
\begin{align*}
\vol(K)\vol(K^*)&\ge \frac{4}{9}\sum_{i=1}^3  \vol_{n-1}(K\cap e^\perp_{i})   \vol_{n-1}\big(P_i(K^*)\big) \\&= \frac{4}{9}\sum_{i=1}^3  \vol_{n-1}(K\cap e^\perp_{i})\vol_{n-1}\big((K\cap e^\perp_{i})^*\big)\\ &\ge  \frac{4}{9}\times 3\times \frac{4^2}{2!}=\frac{4^3}{3!}.
\end{align*}
\begin{flushright}
$ \Box $ \\
\end{flushright}

\subsection{Further  special cases where the conjectures hold}

Let us list here a number of other special cases in which the conjectured inequality was proved:
\begin{itemize}

 \item Symmetric polytopes in $\R^n$ with $2n+2$ vertices  for $n\le 9$  (Lopez and Reisner \cite{LR}) and for any $n$ (Karasev \cite{Ka}). 
\item  For $p\ge 1$,  hyperplane sections through $0$ of $B_p^n=\{(x_1, \dots, x_n)\in\R^n; \sum_{i=1}^n |x_i|^p\le 1\}$  (Karasev
      \cite{Ka}).   Karasev’s proof of those results is, so far, one of the few concrete applications of the symplectic geometry, through billiards approach, to proving special cases of Mahler’s conjecture.
   \item  Bodies of revolution \cite{MR1}.
 \item Some bodies with many symmetries:  Barthe and Fradelizi in \cite{BF}, established that a convex body $K$ which is symmetric with respect to a family of hyperplanes whose intersection is reduced to one point, satisfies  Conjecture \ref{mahler}. More generally, it is proved in \cite{BF} that if $K$ is invariant under the reflections fixing $P_1\times \cdots\times P_k$, where for $1\le i\le k$, the $P_i$ are regular polytopes or an Euclidean ball in a subspace $E_i$ and $\R^n=E_1\oplus\cdots\oplus E_k$, then 
 $\P(K)\ge\P(P_1\times \cdots\times P_k)$.
  \item 
Iriyeh and Shibata established  similar results  in \cite{IS2, IS3}. They determined the exact lower bound of the volume product of convex bodies invariant by some group of symmetries (many classical symmetry groups in dimension 3 \cite{IS2} and for the special orthogonal group of the simplex and of the cube \cite{IS3}). 

 \item Polytopes in $\R^n$ with not more that $n+3$ vertices \cite{MR2}.
 
 \item  Almost unconditional symmetric bodies (Saint Raymond \cite {SR1}) with equality case for Hanner polytopes (Meyer \cite{Me1}, Reisner \cite{Re3}).
 Also in \cite {SR1} is proved a result for unconditional sums of convex bodies :
 For $1\le i\le m$, let $K_i \subset \R^{d_i}$ be convex symmetric bodies and let $L\subset \R^m$ be an unconditional body with respect to  the canonical basis $e_1, \dots, e_m$. We 
 define {\it the unconditional sum of $K_1,\dots, K_m$ with respect to $L$} by
 $$\hskip 15mm
 K_1\oplus_L\dots \oplus_L K_m = \{(x_1,\dots,x_m)\in \R^{d_1}\times\dots\times \R^{d_m};  \| x_1\|_{K_1}e_1+\dots +\| x_m\|_{K_m}e_m \in L\}. 
 $$
 Clearly $K_1\oplus_L\dots \oplus_L K_m$ is a symmetric convex body in $\R^{d_1+\dots+d_m}$. Moreover it is easy to see that
 $\big(K_1\oplus_L\dots \oplus_L K_m \big)^*= K_1^*\oplus_{L^*}\dots \oplus_{L^*} K_m^*$ and denoting $L_+=L\cap \R^m_+$
 and $^*L_+=L^*\cap \R^m_+$, one has 
 $$
 \VP(K_1\oplus_L\dots \oplus_L K_m)= \Big(\int_{(t_1,\dots, t_m) \in L_+}\prod_{i=1}^m t_i^{d_i -1} dt_1\dots dt_m\Big) \times $$ $$\Big(\int_{(t_1,\dots, t_m) \in L^*_+} \prod_{i=1}^m t_i^{d_i -1} dt_1\dots dt_m  \Big) \prod_{i=1}^m \VP(K_i)$$
 and 
 $$\hskip 10mm \Big(\int_{(t_1,\dots, t_m) \in L_+}\prod_{i=1}^m t_i^{d_i-1} dt_1\dots dt_m\Big)   \Big(\int_{(t_1,\dots, t_m) \in L^*_+} \prod_{i=1}^m t_i^{d_i -1} dt_1\dots dt_m  \Big)\ge \frac{d_1 !\times \dots\times d_m !}{(d_1+\dots +d_m)!}\ .$$
 Observe that it follows from \cite{Me1} or \cite {Re3} that there is equality in the last inequality if and only if $L$ is a Hanner polytope.
Finally,  if $\VP(K_i)\ge 4^i/i!$, $1\le i\le m$, then
$$\VP(K_1\oplus_L\dots \oplus_L K_m)\ge \frac{4^{d_1+\dots+d_m}}
{(d_1+\dots+d_m)!}. $$
\item  Although their volumes have been computed (see \cite{SR2}), it is not known whether the unit ball of classical ideals of operators satisfy Conjecture 2.
\item An interpretation of Conjecture 2 in terms of wavelets was given in \cite{Ba2}.

\item Connections of Mahler's conjecture and the Blaschke-Santal\'o inequality to the maximal and minimal of $\lambda_1(K)\lambda_1(K^*)$, where $K$ is a convex body and  $\lambda_1(K)$ is first eigenvalue of the Laplacian on the relative interior of K with Dirichlet condition $u = 0$ on $\partial K$ was given in \cite{BuF}.

\end{itemize}

\subsection{Local minimizers and stability results}

One may investigate the properties of the local minimizers for $\VP(K)$. A natural open question is whether such a minimizer must be a polytope. A number of results in this direction were proved by studying convex bodies with positive curvature. Stancu \cite{St} proved that if $K$ is a convex body, which is smooth enough and has a strictly positive Gauss curvature everywhere, then the volume product of $K$ can not be a local minimum. 
She showed it as a consequence of the fact that, for some $\delta(K)>0$, one has
\[
\vol(K_\delta)\vol((K_\delta)^*)\ge
\vol(K)\vol(K^*)\ge \vol(K^\delta)\vol((K^\delta)^*),
\] 
for any $\delta\in(0,\delta(K))$, where $K_\delta$ and $K^\delta$ stand for the convex floating body and the illumination body associated to $K$ with parameter $\delta$. 
A stronger result for local minimizers was proved in \cite{RSW}:  if $K$ is a convex body which is local minimizer of volume product, then $K$ has no  positive curvature at any point of its boundary. The study of local minimizers was continued in \cite{HHL}, where the authors  computed the first and the second derivative of the volume product in terms of the support function.
 Those results may be seen as a  hint toward
the conjecture that a minimizer must be a polytope. We also note that \cite{GM} extended it to the functional
case (see Section \ref{func} below).

It is known  that the  conjectured global minimizers, that is Hanner polytopes in the centrally symmetric case and simplices in the general case, are actually local minimizers. This question originates from the blog of  Tao \cite{T1, T2}, where a number of ideas that may lead to a better understanding of the volume product were discussed. Nazarov, Petrov, Ryabogin and Zvavitch \cite{NPRZ} were able to show that the cube and the cross-polytope are local minimizers. Kim and Reisner \cite{KiR} generalized this result to the case of non-symmetric bodies proving that the simplex is a local minimizer.

The most general  result in the symmetric case  was obtained by  Kim \cite{Ki}  who considered the case of 
 Hanner polytopes.  More precisely, let 
 $$d_{BM}(K,L)=\inf\{d:\; d>0, \mbox{ there exists } T \in GL(n) \mbox{  such that
 } K\subseteq TL\subseteq d K\}$$ be the Banach-Mazur
multiplicative distance between two symmetric convex bodies $K, L \subset \R^n$. Then
\begin{thm}\label{JKim}
 There exist constants $\delta(n), c(n)>0$  depending  only on $n$  such that if $K$ be a symmetric convex body in $\R^n$ with $$\min \{d_{BM}(K,H): H \mbox{ is a Hanner polytope in } \R^n \} = 1+ \delta,$$ for some $0<\delta  \le \delta (n)$, then
$$
\VP(K)\ge(1+c(n)\delta)\cdot\VP(B_\infty^n).
$$
\end{thm}
 
The above theorem was used in \cite{KiZ} to show the stability of the volume product around the class of unconditional bodies.   The question of stability for minima and  maxima was also treated in various cases \cite{BMMR, BH, KiZ, Bor, FHMRZ}. A  general approach to global stability of the volume product was considered in  \cite{FHMRZ}, where the following natural lemma was proved:
\begin{lemma}\label{lem:metric}
Let $({\mathcal A}_1,d_1)$ be a compact metric space, $({\mathcal A}_2,d_2)$ be a metric space, $f:{\mathcal A}_1\to {\mathcal A}_2$ be a continuous function and $D$ be a closed subset of ${\mathcal A}_2$. Then, 
   
 \noindent(1) For any $\beta>0$, there exists $\alpha>0$, such that $d_1(x,f^{-1}(D))\ge\beta$ implies $d_2(f(x),D)\ge\alpha$.
 
\noindent(2) If  for some $c_1,c_2>0$, $d_1(x,f^{-1}(D))<c_1$ implies $d_2(f(x), D))\ge c_2d_1(x,f^{-1}(D)),$ then  for some $C>0$, one has
$d_1(x,f^{-1}(D)) \le c d_2(f(x), D))$ for every $x \in {\mathcal A}_1.$

 \end{lemma}
 Together with a local minima result (for example Theorem \ref{JKim}),  Lemma \ref{lem:metric} gives almost immediately a stability result for known bounds of the volume product. Let us illustrate this technique in the case of  symmetric convex bodies in $\R^3$.
\begin{thm}\label{thm:stability_BM}
There exists an absolute constant $C>0$, such that for every symmetric convex body $K \subset \R^3$ and $\delta>0$  satisfying
$\Pp(K) \leq (1+ \delta)\Pp(B_\infty^3)$, one has
$$
\min\{ d_{BM}(K, B_\infty^3), d_{BM}(K, B_1^3)\} \le 1+C\delta.
$$
\end{thm}
\begin{proof} Using the linear invariance of the volume product and John's theorem,  we reduce to the case   $B_2^3\subseteq K \subseteq \sqrt{3} B_2^3$. Our metric space ${\mathcal A}_1$ will be  the set of such bodies  with the Hausdorff metric $d_H$. Let ${\mathcal A}_2=\R$.  Then $f:{\mathcal A}_1 \to {\mathcal A}_2$, defined by  $f(K)=\VP(K)$,  is continuous on ${\mathcal A}_1$ (see for example \cite{FMZ}).  Finally, let $D=\VP(B_\infty^3)$. From the description of the equality  cases (i.e. that $K$ or $K^*$ must be  a parallelepiped) proved in \cite{IS1, FHMRZ} we get
\begin{align*}
f^{-1}(D)&=\{K\in {\mathcal A}_1; \Pp(K)=\Pp(B_\infty^3)\}\\ &=\{K\in {\mathcal A}_1; K=S B_\infty^3\ \hbox{or}\ K=\sqrt{3}SB_1^3, \mbox{ for some } S\in {\rm SO}(3) \}.
\end{align*}
Note that $B_\infty^3$ is in John position (see for example \cite{AGM1}) and thus  if $ B_2^3 \subset T B_\infty^3 \subset \sqrt{3}B_2^3$ for some $T \in GL(3)$, then $T \in SO(3)$.

Next,we show that the assumptions in the  second part of Lemma \ref{lem:metric} are satisfied.  Since $d_{BM}(K^*, L^*)= d_{BM}(K, L)$,  we may restate the $\R^3$  version of Theorem \ref{JKim}  in the following form: there are absolute constants $c_1, c_2 >0$ such that  for every  symmetric convex body $K$ in $\R^3$ satisfying
$\min\{ d_{BM}(K, B_\infty^3), d_{BM}(K, B_1^3)\}:=1+d \le  1+c_1,$
one has
$
\VP(K) \ge  \VP(B_\infty^3)+ c_2 d.
$
To finish checking the assumption, note that for all $K,L$ convex bodies such that $B_2^3\subseteq K,L \subseteq \sqrt{3} B_2^3$, one has:
\begin{eqnarray}\label{eq:dist}
d_{BM}(K, L)-1 \le   \min_{T\in GL(3)} d_H (TK, L) \le  \sqrt{3}(d_{BM}(K, L) -1 ). 
\end{eqnarray}
Applying Lemma \ref{lem:metric}, we deduce that there exists $c>0$ such that if $B_2^3\subseteq K \subseteq \sqrt{3} B_2^3$, then
\[
\min_{S\in SO(3)}\min(d_H(K,SB_\infty^3), d_H(K,S\sqrt{3}B_1^3))\le c|\VP(K)-\VP(B_\infty^3)|.
\]
Using (\ref{eq:dist}) we conclude the proof.

\end{proof}

\section{Asymptotic estimates and Bourgain-Milman's theorem} \label{AE}

If Conjecture 2 holds true for centrally symmetric bodies $K$, then  one has 
$$ \frac{4} {n!^{\frac{1}{n} } }  \le \VP(K)^{\frac{1}{ n} }\le \frac{\pi} {  \Gamma(1+\frac{n}{2}  )^{\frac{2}{n}} }\ ,$$
so that 
 \begin{equation}\label{BM}
 \frac{4e+o(1)}{n}\le \VP(K)^{\frac{1}{n}} \le  \frac{2e\pi+o(1)}{n}.
 \end{equation}
Similarly, the truth of Conjecture \ref{mahler} would imply that for any convex body $K$, one has   
\[\P(K)^\frac{1}{n}\ge \VP(\Delta_n)^\frac{1}{n}\ge  \frac{e^2+o(1)}{n}.
\]
 So that the function $K\mapsto n\VP(K)^{\frac{1}{ n} }$ would vary between two positive constants.
This last fact was actually proved  by Bourgain and Milman \cite{BM} in 1986. Indeed,  the  upper bound is insured by the  Blaschke-Santal\'o inequality. For the lower bound, the first important step  was done by Gordon and Reisner \cite{GR}, who proved that 
 $$\VP(K)^{\frac{1}{ n} }\ge \frac{c}{n\log(n)}\ .$$
 Then, Bourgain and Milman \cite{BM}  proved that 
 \begin{equation}\label{BourgMil}
 \VP(K)^{\frac{1}{n}}\ge \frac{c}{n}.
\end{equation}
For the original proof of (\ref{BourgMil})  and other proofs of the same type, see  \cite{BM, LMi, Pi}. The constant $c$ obtained in those proofs was not at all explicit, and even if so, 
was quite small. After having given a low technology proof of Gordon-Reisner result \cite{Ku1}, G. Kuperberg \cite{Ku2} gave another proof of \eqref{BourgMil} based on differential geometry, and got the explicit constant $c=\pi e$ in \eqref{BourgMil} in the symmetric case, which is not far from the best possible bound  $4e$ and  is the best constant known for now. The best constant in the general (i.e not necessary symmetric) case may be obtained using Rogers-Shephard inequality, see the end of this section.  Using Fourier transform techniques, other proofs were given by Nazarov \cite{Na} (see also  
 Blocki  \cite{Blo1, Blo2}, Berndtsson\cite{Be1, Be2} and Mastroianis  and Rubinstein \cite{MaR}).  Giannopoulos, Paouris and Vritsiou gave also a proof using classical techniques of the  local theory of Banach spaces  \cite{GPV}.
 
 The   isomorphic version of the lower bound in (\ref{BM})  is "the best possible step" one can make, before actually proving (or disproving) the Mahler conjecture.  Indeed, assume we can achieve an asymptotic behavior better than $\VP(K) \ge c^n \VP(B_\infty^n)$, $0<c<1$,    i.e. we have  
\begin{equation}\label{BMimprove}
\alpha(n)\VP(B_\infty^n) \le \VP(K), \mbox{ and } \lim\limits_{n\to\infty}\alpha(n)/c^n=\infty,
\end{equation}
but there is a dimension, say $l$, such that the Mahler conjecture is false in $\R^l$, i.e.  there exists a convex symmetric body $K \subset \R^l$ such that $\VP(K) < \VP(B_\infty^l)$ or
$$
\VP(K) \le c_2 \VP(B_\infty^l), \mbox{ for some  } 0<c_2<1.
$$
 Let  $K'$  to be the $m$-th  direct sum of copies of $K$, $K'=K\oplus  \dots \oplus K\subset \R^{n}$, $n=ml$,     using the direct sum formula  and  (\ref{BMimprove}) inequality we get
$$
\alpha(lm)\VP(B_\infty^{l m})\le\VP(K')=\VP(K\oplus  \dots \oplus K) \le c_2^m \VP(B_\infty^{l m})=(c_2^{1/l})^{lm} \VP(B_\infty^{l m}).
$$
 This yields $\alpha(n)\le c^n$ for $n=ml$ and $c=c_2^{1/l}$,  and  we get a contradiction  for $m$ big enough with $\lim\limits_{n\to\infty}\alpha(n)/c^n=\infty$.

 We note that (\ref{BourgMil}) for  general convex bodies follows (with a constant divided by two) from the symmetric case. Indeed,  let $L$ be a convex body in $\R^n$ and let $z\in \inte(L)$.  Let $K=L-z$. Then by the Rogers-Shephard inequality \cite{RS}, $\vol(\frac{K-K}{2})\le 2^{-n}\binom{2n}{n} \vol(K)\le 2^n \vol(K)$ and 
\[
\vol\left(\left(\frac{K-K}{2}\right)^*\right)=\frac{1}{n}\int_{S^{n-1}}\!\!\left(\frac{h_K(u)+h_{-K}(u)}{2}\right)^{-n}\!\!\!d\sigma(u)\le \frac{1}{n}\int_{S^{n-1}}\!\!h_K(u)^{-n}d\sigma(u)=\vol(K^*).
\]
It follows that
 $$
 \vol(K) \vol(K^*)\ge 2^{-n}\VP\left(\frac{K-K}{2}\right).
 $$
 Since this holds for every $z\in\inte(L)$, it follows that 
 $\VP(L)\ge 2^{-n} \VP\left(\frac{L-L}{2}\right)$.
From this relation and Kuperberg's best bound $c=\pi e$ in \eqref{BourgMil} for  symmetric bodies, it follows that for general convex bodies,  \eqref{BourgMil} holds with $c=\pi e/2$.

 \subsection{Approach via Milman's quotient of subspace's theorem} 

 The next lemma is a consequence of the Rogers-Shephard inequality \cite{RS}.

 \begin{lemma}\label{RS}
 Let $K$ be a convex symmetric body in $\R^n$, let $E$ be an $m$-dimensional subspace of $E$ and $E^{\perp}$ be its orthogonal subspace. Then
 $$  \binom{n}{m}^{-2}\VP(K\cap E)  \VP(K\cap E^{\perp})\le   \VP(K)\le \VP(K\cap E)  \VP(K\cap E^{\perp}).$$
  \end{lemma}

The following result is the {\it quotient of subspace theorem} of V. Milman (\cite{Mi}, see \cite{Gor} for a simple proof).
\begin{thm}\label{Milman} Let  $K$ be a convex symmetric body in $\R^n$, with $n$ a multiple of $4$. Then, there exists a constant $c>0$, independent on $n$, a $\frac{n}{2}$-dimensional subspace $E$ of $\R^n$, a $\frac{n}{4}$-dimensional subspace $F$ of $E$ and an ellipsoid ${\mathcal E}\subset F$ such that
  $$  {\mathcal E}\subset  P_F(E\cap K) \subset   c {\mathcal E}$$
  where, as before,  $P_F$ is the orthogonal projection onto $F$.
    \end{thm}
    \noindent
{\it  The proof of Bourgain-Milman's theorem by Pisier \cite{Pi}.}
For a convex symmetric body $K\subset \R^n$, with $n$ multiple of $4$, let $a_n(K)=n\VP(K)^{\frac{1}{n}}$. Let $E$  and $F$ be  the subspaces of $\R^n$ 
  chosen in Theorem \ref{Milman}. By lemma \ref{RS}, 
  for some constant $d>0$ independent on $n$,  one has
 $$d\sqrt{a_{\frac{ n}{2}}( K\cap E) a_{\frac{ n}{2}}( K\cap E^{\perp})}    \le a_n(K)\le \sqrt{a_{\frac{ n}{2}}( K\cap E) a_{\frac{ n}{2}}( K\cap E^{\perp})} $$
and

$$
d\sqrt{a_{\frac{ n}{4}}( P_F(K\cap E)) a_{\frac{ n}{4}}( K\cap E\cap F^{\perp}} )       \le       a_{\frac{ n}{2}}( K\cap E) 
\le \sqrt{a_{\frac{ n}{4}}( P_F(K\cap E)) a_{\frac{ n}{4}}( K\cap E\cap F^{\perp}} ).
$$
Next, from Theorem \ref{Milman}, for some absolute constants $c',d'>0$, one has 
$$
c'\le a_{\frac{ n}{4}}( P_F(K\cap E))\le d'.
$$
It follows that   for some universal constant $c>0$, one has
\begin{equation}\label{fund}
a_n(K)\ge 
c \big(a_{\frac{ n}{2}}( K\cap E^{\perp})   \big)^{\frac{1}{2}} \big(a_{\frac{n}{4} }( K\cap E\cap F^{\perp}) \big)^{\frac{1}{4}}.
\end{equation}
Defining now for every $n\ge 1$, 
$$
a_n=\min\{a_m(L); 1\le m\le n, \mbox{ $L$ convex symmetric body in $ \R^m$} \}.
$$
Observing that $a_n>0$, one gets from (\ref{fund}) that 
\begin{equation}\label{fin}
a_n\ge c \big(a_n\big)^{\frac{1}{2}} \big(a_n\big)^{\frac{1}{4}}.
\end{equation}
Thus $a_n\ge c^4$.

 \subsection{Complex analysis approach} 
 Let us very briefly  discuss an approach via complex and harmonic analysis which was initiated by  Nazarov \cite{Na}. We will follow here a work of Berndtsson \cite{Be2}, which  is done via  functional inequalities (central for the next section). 
 
We will consider a special case of the Bergman spaces. Let $\psi: {\mathbb C}^n \to \R\cup\{+\infty\}$ be a convex function and $\Omega=\{(x,y)\in \R^{2n}: \psi(x+iy)<\infty\}.$ The {\it Bergman space} $A^2(e^{-\psi})$ is the Hilbert space of holomorphic functions $f$  on $\Omega$ such that 
$$\|f\|^2=\int_{\Omega} |f(x+iy)|^2 e^{-\psi(x+iy)} dxdy <\infty.
$$
 The (diagonal) Bergman kernel $B$ for $A^2(e^{-\psi})$ is defined as
 $$
 B(z)=\sup_{f \in A^2( e^{-\psi})} \frac{|f(z)|^2}{\|f\|^2}.
 $$
 Next, consider an even convex function $\phi:\R^n \to\R\cup \{+\infty\}$ such that $ e^{-\phi(x)}$ is integrable over $\R^n$. For  $\alpha \in {\mathbb C}$, consider the Bergman kernel $B_\alpha(z)$ corresponding to the function $\psi(z)=\phi({\rm Re}(z))+ \phi({\rm Re}(\alpha z))$.  The main theorem in \cite{Be2} is  the claim that
 \begin{equation}\label{berndtsson}
 B_i(0) \le c^n B_1(0),
 \end{equation}
 where $c$
is an absolute constant (precisely computed in \cite{Be2}). We note that $B_1$ is the Bergman kernel for $\psi(x+iy)=2\phi(x)$, i.e. independent of the ${\rm Im}(z)$, and that $B_i$ is the Bergman kernel for $\psi(x+iy)=\phi(x)+\phi(y)$.  It is essential to understand that the Bergman spaces corresponding to those densities are different. Thus the connection is not immediate. For example,  the function $f=1$  belongs to the second space but does not belong to the space corresponding to $\psi(x+iy)=2\phi(x)$. Using $f=1$ we get    
$$
B_i(0)\ge \frac{1}{\int_{\R^n} e^{-\phi(x)} dx \int_{\R^n} e^{-\phi(y)}dy}\ .
$$
Together with (\ref{berndtsson}) this gives 
 \begin{equation}\label{berndtssonM}
 B_1(0) \ge  \frac{c^{-n}}{\left(\int_{\R^n} e^{-\phi(x)} dx \right)^2}\ ,
 \end{equation}
 which is an essential estimate for proving the Bourgain-Milman inequality.
 The proof of (\ref{berndtsson}) in \cite{Be2} is based on a very nice and tricky approach of "linking" $B_i$ and $B_1$ via $B_\alpha$. Indeed. it turns out that $b(\alpha):=\log B_\alpha(0)$ is subharmonic in ${\mathbb C}$ (see \cite{Be0, Be2}), and moreover $b(\alpha) \le C+n \log|\alpha|^2, $ which can be seen from the change of variables
 \begin{align*}
 \|f\|_\alpha^2&=\int_{{\mathbb C}^n}|f(z)|^2  e^{-(\phi({\rm Re}(z)))+ \phi({\rm Re}(\alpha z)))} dz \\ &=|\alpha|^{-2n} \int_{{\mathbb C}^n}|f(z/\alpha)|^2  e^{-(\phi({\rm Re}(z/\alpha)))+ \phi({\rm Re}(z)))} dz.
 \end{align*}
  Thus $B_\alpha(0)=|\alpha|^{2n}B_{1/\alpha}(0).$ Moreover, $B_{1/\alpha}(0)$ is bounded  as $\alpha \to \infty.$   Thus one can apply the Poisson representation formula in the upper half plane to the function $b(\alpha)-n\log|\alpha|^2$ to get 
$$
\log B_i(0)=b(i)\le \frac{1}{\pi}\int_{-\infty}^\infty \frac{b(s) - n \log(s^2)}{1+s^2} ds=\frac{2}{\pi}\int_{0}^\infty \frac{b(s) - n \log(s^2)}{1+s^2} ds.
$$
Using that  $s\mapsto \phi(s x)$ is increasing on  $(0, 1]$ one has, for $s\in (0, 1]$, 
  $$
\|f\|_s^2=\int_{{\mathbb C}^n}|f(z)|^2  e^{-(\phi({\rm Re}(z))+ \phi({\rm Re}(s z))} dz \ge  \|f\|_1^2,
 $$
 and hence $b(s) \le b(1)$. If $s \ge 1$,
 $$
 \|f\|_s^2 =|s|^{-2n} \int_{{\mathbb C}^n}|f(z/s)|^2  e^{-(\phi({\rm Re}(z/\alpha))+ \phi({\rm Re}(z))} dz \ge s^{-2n}\|f\|_1^2,
 $$
 thus $b(s)\le b(1)+n\log s^2$. Putting those estimates together completes the proof of
 (\ref{berndtsson}).

 The next step is to adapt the Paley-Wiener space associated to a convex body (discussed in Theorem \ref{PW}) to the case of a convex function. For a convex function $\varphi: \R^n \to \R \cup \{+\infty\}$, we denote by $PW(e^\varphi)$ the space of holomorphic functions $f$ of the form 
 $$f(z)=\int_{\R^n} e^{\langle z,  \xi \rangle} \tilde{f}(\xi) d\xi, \hbox{  where }z \in {\mathbb C}^n, $$
 for which
 $$
 \|f\|_{PW}^2 =\int_{\R^n} |\tilde{f}|^2 e^{\varphi} dt < \infty,
 $$
 for some function $\tilde{f}$, so that the two formulas above make sense.
 The classical Paley-Wiener space discussed in  Theorem \ref{PW} then corresponds to the case when $\varphi(x)=0$ for $x\in K$ and $\varphi(x)=+\infty$ for $x \not \in K.$ For a convex function $\psi$ on $\R^n$, let us consider its logarithmic Laplace transform given by
 $$
 \Lambda \psi (\xi)= \log \int_{\R^n} e^{2\langle x, \xi \rangle} e^{-\psi}dx.
 $$
The second key ingredient in Berndtsson's proof is the fact that the spaces $PW(e^{\Lambda \psi})$ and $A^2(e^{-\psi})$ coincide and  that \begin{equation}\label{normequality}\|f\|_{A^2}^2=(2\pi)^{2n}\|f\|_{PW(e^{\Lambda \psi})}^2.
\end{equation}
This fact originates from the observation that any  $f\in PW(e^{\Lambda \psi})$ is the Fourier-Laplace transform of $\tilde{f}$ and $e^{\langle x, t \rangle}\tilde{f}(t)$ belongs to $L_2(\R^n)$ for all $x$ such that $\psi(x)<\infty$. Then, we apply Parseval's formula to get 
$$
\int_{\R^n}|f(x+iy)|^2 dy=(2\pi)^n \int_{\R^n} e^{2\langle x, t \rangle}|\tilde{f}(t)|^2dt.
$$
Multiplying the above equality by $e^{-\psi(x)}$ and integrating with respect to $x$, we get
$$
\int_{\R^n}\int_{\R^n}|f(x+iy)|^2e^{-\psi(x)}dxdy=(2\pi)^n \int_{\R^n} |\tilde{f}(t)|^2 e^{\Lambda \psi(t)}dt.
$$
Thus $f \in A^2(e^{-\psi})$ and the $A^2$ norm coincide with a multiple of the norm in $PW(e^{\Lambda \psi})$. This confirms that the Paley-Wiener space is isometrically embedded into the corresponding Bergman space and the rest follows from the observation that it is dense.

One can compute the Bergman kernel for $PW(e^{\Lambda \psi})$ and use (\ref{normequality}) to show that the Bergman kernel for $A^2(e^{-\psi})$ is equal to
\begin{equation}\label{PWestimate}
(2\pi)^{-n}\int_{\R^n} e^{2\langle x, t \rangle - \Lambda \psi(t)}dt.
\end{equation}
We will use (\ref{PWestimate}) to give an estimate from above  of the value of the Bergman kernel at zero.  The Legendre transform $\L\psi$ of  a function $\psi:\R^n\to\R\cup\{+\infty\}$ is defined by
\begin{equation}\label{legan}
\L\psi(y)=\sup_{x\in \R^n}(\langle x,y\rangle -\psi(x)),\quad \mbox{ for $y\in \R^n$}.
\end{equation}
Consider the Bergman  space $A^2(e^{-2\phi(x)}),$ where $\phi:\R^n \to \R \cup \{\infty\}$ is convex and even (as in (\ref{berndtssonM})). Then,
\begin{equation}\label{eqmain}
B(0) \le \pi^{-n}\frac{\int_{\R^n} e^{-\L\phi(y)}dy}{\int_{\R^n}e^{-\phi(x)}dx}.
\end{equation}
Indeed, using (\ref{PWestimate}) we get 
\begin{equation}\label{eqb0}
B(0) \le (2\pi)^{-n}\int_{\R^n} e^{ - \Lambda (2\phi(t))}dt.
\end{equation}
Note that, for any $y\in \R^n$, one has
\begin{align*}
    e^{  \Lambda (2\phi(t))}&=2^{-n}\int_{\R^n}e^{\langle t,u\rangle - 2 \phi(u/2) }du = 2^{-n} e^{\langle t,y\rangle} \int_{\R^n}e^{\langle t,v\rangle - 2 \phi(v/2+y/2) }dv\\
    &\ge 2^{-n} e^{\langle t,y\rangle-\phi(y)} \int_{\R^n}e^{\langle t,v\rangle -  \phi(v) }dv,
\end{align*}
where in the last inequality we used the convexity of $\phi$. Using that $\phi$ is  even, we get that
$$\int_{\R^n}e^{\langle t,v\rangle -  \phi(v) }dv \ge \int_{\R^n}e^{-  \phi(v) }dv$$
 and 
$$
e^{  \Lambda (2\phi(t))}\ge 2^{-n} e^{\langle t,y\rangle-\phi(y)} \int_{\R^n}e^{-  \phi(v) }dv.
$$
Taking the supremum over all $y\in \R^n$, we get 
$$
e^{  \Lambda (2\phi(t))}\ge 2^{-n} e^{\L\phi(t)} \int_{\R^n}e^{-  \phi(v) }dv.
$$
Together with (\ref{eqb0}), this gives (\ref{eqmain}). Combining (\ref{eqmain}) with (\ref{berndtssonM}), we get the following theorem, 
\begin{thm}\label{fbm}{\bf (Functional version of the Bourgain-Milman inequality)} Let $\phi:\R^n \to \R \cup \{+\infty\}$ be  even and convex; then for some $c>0$ independant on $n$, one has
$$
\int_{\R^n} e^{  - \phi(x)} dx  \int_{\R^n} e^{  - \L\phi(x)} dx\ge c^n.
$$
\end{thm}

\begin{remark} Theorem \ref{fbm} was first proved, via Bourgain-Milman inequality for symmetric convex bodies in \cite{AKM} and then generalized to non-even functions in \cite{FM3}. It implies the classical Bourgain-Milman's inequality for convex bodies as we shall see in the next section (Remark \ref{bodyfun} below).
\end{remark}

\section{Functional inequalities and link with transport inequalities}\label{func}

We dedicate this section  to the study of functional inequalities related to volume product.

\subsection{Upper bounds}

The following general form of the functional Blaschke-Santal\'o inequality was proved by Ball \cite{Ba0} for $f$ even, by Fradelizi and Meyer \cite{FM4} for $f$ log-concave and by Lehec \cite{Le1} in the general case.  

\begin{thm}\label{thm:bs-func-gen}
Let $f: \R^n \to \R_+$ be Lebesgue integrable. There exists $z \in \R^n$ such that for any $\rho: \R_+ \to \R_+$ and any $g:\R^n \to \R_+$  measurable satisfying $$f(x+z)g(y) \leq \rho(\langle x,y \rangle)^2\mbox{  for all $x,y\in \R^n$ satisfying $\langle x,y \rangle >0$} ,$$
one has 
  $$
\int f(x)\,dx \int g(y)\,dy \leq \left(\int \rho(|x|^2)\,dx\right)^2 .
$$
If $f$ is even, one can take $z=0$.
\end{thm}

Applying this result to $\rho={\bf 1}_{[0,1]}$ and $f={\bf 1}_K$, one recovers the Blaschke-Santal\'o inequality for convex sets. 
 Applying it to $\rho(t)=e^{-t/2}$, it gives a proof of the following functional Blaschke-Santal\'o inequality for the Legendre transform due to Artstein, Klartag and Milman \cite{AKM} (and \cite{Le2} for another proof). 

\begin{thm}\label{thm:bs-func} Let $\varphi:\R^n\to\R\cup\{+\infty\}$ satisfy $0<\int e^{-\varphi}<+\infty$. If for $x,y\in \R^n$,  $\varphi_y(x):=\varphi(x+y)$,  there exists $z\in\R^n$ such that
  $$
\int_{\R^n}e^{-\varphi(x)}dx\int_{\R^n}e^{-\L(\varphi_z)(y)}dy\le\left(\int_{\R^n}e^{-\frac{|x|^2}{2}}dx\right)^2=(2\pi)^n,
$$
with equality if and only  $\varphi_z(x)=|Ax|^2$ for some  invertible linear map $A$ and some $z\in \R^n$.
\end{thm}

\begin{remark}\label{rk:lehec-centered}
In \cite{Le2}, Lehec deduced from  Theorem \ref{thm:bs-func}  that if the "barycenter" $b(\varphi):=\int xe^{-\varphi(x)}dx/\int e^{-\varphi}$ satisfies $b(\varphi)=0$, then  
$$
\int e^{-\varphi}\int e^{-\L\varphi}\le(2\pi)^n.
$$ 
Indeed, for any $z$, one has $\L(\varphi_z)(y)=\L\varphi(y)-\langle y,z\rangle$. It follows that $\L((\L\varphi)_z)(y)=\L\L\varphi(y)-\langle y,z\rangle\le\varphi(y)-\langle y,z\rangle$. Using  Jensen's inequality and $b(\varphi)=0$, we get
\[
\int e^{-\L((\L\varphi)_z)}\ge\int e^{-\varphi(y)+\langle y,z\rangle}dy\ge e^{\langle b(\varphi),z\rangle}\int e^{-\varphi}=\int e^{-\varphi}.
\]
Applying Theorem \ref{thm:bs-func} to $\L\varphi$, there exists thus  a $z$ such that
\[
\int e^{-\varphi}\int e^{-\L\varphi}\le \int e^{-\L((\L\varphi)_z)}\int e^{-\L\varphi}\le(2\pi)^n.
\]
\end{remark}
As Lehec observed also, this gives a new proof of the result of Lutwak \cite{Lu91}:

\begin{proposition}\label{propLut} For starshaped body $K\subset\R^n$ (for all $(x,t) \in K\times[0,1]$, one has $tx\in K$) with barycenter at $0$, one has  $$\vol(K)\vol(K^*)\le\vol(B_2^n)^2.$$
\end{proposition}
\begin{proof} Let $\varphi(x)=\frac{\|x\|_K^2}{2}$. Then   since 
\[
\int_{\R^n}xe^{-\frac{\|x\|_K^2}{2}}dx=\int_{\R^n}x\int_{\|x\|_K}^{+\infty}te^{-\frac{t^2}{2}}dtdx=\int_0^{+\infty}t^{n+1}e^{-\frac{t^2}{2}}dt\int_K xdx=0,
\]
one has $b(\varphi)=0$. Moreover for any $y\in\R^n$, one has $\L\varphi(y)=\sup_x\langle x,y\rangle -\frac{\|x\|_K^2}{2}=\frac{\|y\|_{K^*}^2}{2}$ and 
 $\int_{\R^n}e^{-\frac{\|x\|_K^2}{2} }dx=
2^{\frac{n}{2}} \Gamma(\frac{n}{2}+1)\vol(K).$ 
\end{proof}
Before giving sketches of various proofs of Theorems \ref{thm:bs-func-gen} and \ref{thm:bs-func},   we need a lemma:
\begin{lemma}\label{lem:PLmult}
Let $\alpha,\beta, \gamma:\R_+\to\R_+$ be measurable functions such that for every $s,t>0$ one has $\alpha(s)\beta(t)\le \gamma(\sqrt{st})^2$, then 
  $
\int_{\R_+}\alpha(t)dt\int_{\R_+}\beta(t)dt\le\left(\int_{\R_+}\gamma(t)dt\right)^2.
$
\end{lemma}
\begin{proof}
Define  $f,g,h:\R\to\R$ by $f(x)=\alpha(e^x)e^x$, $g(x)=\beta(e^x)e^x$ and $h(x)=\gamma(e^x)e^x$. Then $f(x)g(y)\le h(\frac{x+y}{2})$ for all $x,y\in\R$. By Prékopa-Leindler inequality (see \cite{Pi}, p.3)  we get
  $
\int_{\R}f(x)dx\int_{\R}g(x)dx\le\left(\int_{\R}h(x)dx\right)^2.
$
We conclude with a change of variables.
\end{proof}
\vskip 2mm
\noindent{\bf Proofs of Theorem \ref{thm:bs-func-gen}:} \\
1) In the case when $f$ is even and $\rho$ is decreasing,  this proof is due to Ball \cite{Ba0}. For $s,t\in\R_+$, let $K_s=\{f\ge s\}$ and $L_t=\{g\ge t\}$. The hypothesis on $f$ and $g$ implies that   $L_t\subset \rho^{-1}(\sqrt{st})K_s^*$. Since $f$ is even, $K_s$ is symmetric. We deduce from Blaschke-Santal\'o inequality that for every $s,t\in\R_+$, if $\alpha(s)=\vol(K_s)$ and $\beta(t)=\vol(L_t)$, one has
  $$ 
\alpha(s)\beta(t)=\vol(K_s)\vol(L_t) \le(\rho^{-1}(\sqrt{st}))^n\vol(K_s)\vol(K_s^*)\le(\rho^{-1}(\sqrt{st}))^n\vol(B_2^n)^2.
$$
Denoting $\gamma(t)=(\rho^{-1}(t))^{n/2}\vol(B_2^n)$, we apply Lemma \ref{lem:PLmult} and 
use that $\int_{\R^n} f(x) dx=\int_0^{+\infty}\alpha(s) ds $... to conclude.

\vskip 1mm
\noindent 
2) In the case when $f$ is not supposed to be  even, but is log-concave, the proof of Theorem \ref{thm:bs-func-gen} given in \cite{FM4} uses the so-called Ball's body $K_f(z)$ associated to a log-concave function $f$,  which is defined by 
  $$
K_f(z)=\left\{x\in\R^n; \int_0^{+\infty} r^{n-1}f(z+rx)dr\ge1\right\}.
$$
It follows from Ball's results \cite{Ba00} that $ K_f(z)$ is convex 
 and that its radial function is $r_{K_f(z)}(x)= \left(\int_0^{+\infty} r^{n-1}f(z+rx)dr\right)^{\frac{1}{n}}$ for  $x\in\R^n\setminus\{0\}$. 
If $x,y\in\R^n$ satisfy $\langle x,y\rangle>0$,  define for $r\ge0$, $\alpha(r)=r^{n-1}f(z+rx)$, $\beta(r)=r^{n-1}g(rx)$ and $\gamma(r)=r^{n-1}\rho(r^2\langle x,y\rangle)$. It follows from Lemma
\ref{lem:PLmult} that
  $$
\int_0^{+\infty} r^{n-1}f(z+rx)dr\int_0^{+\infty} r^{n-1}g(rx)dr\le\left(\int_0^{+\infty}r^{n-1}\rho(r^2\langle x,y\rangle)dr\right)^2.
$$
This means that 
$$
\langle x,y\rangle\le \frac{ c_n(\rho)}{r_{K_f(z)}(x)r_{K_g(0)}(y)} ,
\mbox{ where }c_n(\rho):=\left(\int_0^{+\infty}r^{n-1}\rho(r^2)dr\right)^{2/n},
$$
or in other words $K_g(0)\subset c_n(\rho) K_f(z)^*$. Moreover, one has 
$$
\int_{\R^n} f(x) dx=n\vol\big(K_f(z)\big)\mbox{ 
 for every }z\in \supp(f).
$$ 
Using Brouwer's fixed point theorem, it was proved in \cite{FM4} that for some  $z\in \R^n$, the center of mass of $K_f(z)$ is at the origin. The result  follows  then from Blaschke-Santal\'o inequality. 

This method was also used in \cite{BBF} to prove stability versions of the functional forms of Blaschke-Santal\'o inequality.\\
 
\noindent{\bf Proofs of Theorem \ref{thm:bs-func}:}

\noindent 1) The proof given in \cite{AKM} attaches to $\varphi:\R^n\to\R\cup\{+\infty\}$, {\it supposed here to be even},  the functions $f_m(x)=\left(1-\frac{\varphi(x)}{m}\right)_+^m$, for $m\ge1$ and the convex bodies 
 $$
 K_m(f_m):=\{(x,y)\in\R^{n+m};|y|\le f_m(\sqrt{m}x)^{1/m}\}.
 $$
 When $m\to+\infty$, $f_m\to e^{-\varphi}$ and
 $$
 m^\frac{n}{2}\frac{\vol\big(K_m(f_m)\big)}{\vol(B_2^m)}=\int_{\R^n}f_m(x)dx\to \int_{\R^n}e^{-\varphi(x)}dx.
 $$
 Moreover $K_m(f_m)^*=K_m(\L_mf_m)$, where 
   $
 \L_m(f_m)(y)=\inf_{x}\frac{\left(1-\frac{\langle x,y\rangle}{m}\right)^m_+}{f_m(x)}.
 $
Also, when $m\to+\infty$
 $$\L_m(f_m)(y)\to e^{-\L\varphi(y)}\mbox{ and }
   m^\frac{n}{2}\frac{\vol\big(K_m(\L_m\varphi)\big)}{\vol(B_2^m|}=\int_{\R^n}\L_mf_m(x)dx\to \int_{\R^n}e^{-\L\varphi(x)}dx.
 $$
One then applies the Blaschke-Santal\'o inequality to the bodies $K_m(f_m)$.\\

\noindent
2) Lehec's proof \cite{Le2} of Theorem \ref{thm:bs-func}  uses induction on the dimension. For $n=1$, choose $z\in \R$ such that $\int_z^{+\infty}e^{-\varphi(t)}dt=\int_{\R}e^{-\varphi(t)}dt/2$. For  all $s,t\ge0$, one has $\varphi_z(s)+\L(\varphi_z)(t)\ge st$. Thus, the functions $\alpha(s)=e^{-\varphi_z(s)}$, $\beta(t)=e^{-\L(\varphi_z)(t)}$ and $\gamma(u)=e^{-u^2/2}$ satisfy $\alpha(s)\beta(t)\le \gamma(\sqrt{st})^2$, for every $s,t\ge0$. It follows from Lemma \ref{lem:PLmult} that 
\begin{equation}\label{pl-func-dim1}
\int_0^{+\infty}e^{-\varphi_z(t)}dt\int_0^{+\infty}e^{-\L(\varphi_z(t))}dt=\int_{\R_+}\alpha(t)dt\int_{\R_+}\beta(t)dt\le\left(\int_{\R_+}\gamma(u)du\right)^2=\frac{\pi}{2}.
\end{equation}
This inequality  also holds on $\R_{-}$; adding the two inequalities, we get the result. 

 Now  suppose that the results holds  for $n$, and let us do the induction step. Let $\varphi:\R^{n+1}\to\R\cup\{+\infty\}$. If $X\in\R^{n+1}$, we denote $X=(x,s)\in\R^n\times \R$. Let   
$$\VP(\varphi):=\min_z\int_{\R^{n+1}}e^{-\varphi(X)}dX\int_{\R^{n+1}} e^{-\L(\varphi_z)(X)}dX.$$ 
 For any invertible affine map $A$, one has $\VP(\varphi\circ A)=\VP(\varphi)$. Translating $\varphi$ in the $e_{n+1}$ direction, we may assume that $$\int_{s>0}\int e^{-\varphi(x,s)}dxds=\int_{s<0}\int e^{-\varphi(x,s)}dxds.$$ Define  $b_+(\varphi)$ and $b_-(\varphi)$ in $\R^{n+1}$ by
\[
b_+(\varphi)=\frac{\int_{s>0}\int(x,s)e^{-\varphi(x,s)}dxds}{\int_{s>0}\int e^{-\varphi}dxds}\quad \hbox{and}\quad b_-(\varphi)=\frac{\int_{s<0}\int(x,s)e^{-\varphi(x,s)}dxds}{\int_{s<0}\int e^{-\varphi}dxds}.
\] 
Since $\langle b_+(\varphi),e_{n+1}\rangle>0$ and $\langle b_-(\varphi),e_{n+1}\rangle<0$, the point $\{z\}:=[b_{-}(\varphi),b_+(\varphi)]\cap e_{n+1}^\bot$ is well defined. By translating $\varphi$ in the remaining directions, we may assume that $z=0$.
Let $A$ be the linear invertible map defined by $Ax=x$ for $x\in\R^n$ and $Ae_{n+1}=b_+(\varphi)$. Then $b_+(\varphi\circ A)=A^{-1}b_+(\varphi)=e_{n+1}$. Changing $\varphi$ into $\varphi\circ A$, we may assume that $b_+(\varphi)=e_{n+1}$. We define $\Phi, \Psi:\R^n\to\R\cup\{+\infty\}$ by 
\[
e^{-\Phi(x)}=\int_0^{+\infty}e^{-\varphi(x,t)}dt\quad \hbox{and}\quad
e^{-\Psi(x)}=\int_0^{+\infty}e^{-\L\varphi(x,t)}dt.
\]
Since  $b_+(\varphi)=e_{n+1}$, we get 
$\int_{\R^n}xe^{-\Phi(x)}dx=\int_{t>0}\int_{\R^n}xe^{-\varphi(x,t)}dxdt=0.$
Hence $b(\Phi)=0$. From the induction hypothesis and the remark after  Theorem \ref{thm:bs-func}, it follows that 
\begin{equation}\label{induction-func}
\int_{\R^n}e^{-\Phi(x)}dx\int_{\R^n}e^{-\L\Phi(y)}dy\le (2\pi)^n.
\end{equation}
For every $x,y\in\R^n$ and $s,t\in\R$, let $\varphi^x(s)=\varphi(x,s)$ and $(\L\varphi)^y(t)=\L\varphi(y,t)$. Applying again Lemma \ref{lem:PLmult} as in \eqref{pl-func-dim1}, we get 
\[
\int_0^{+\infty}e^{-\varphi^x(s)}ds\int_0^{+\infty}e^{-\L(\varphi^x)(t)}dt\le\frac{\pi}{2}.
\]
Since $\varphi^x(s)+(\L\varphi)^y(t)\ge\langle x,y\rangle +st$, one has  $(\L\varphi)^y(t)-\langle x,y\rangle\ge \L(\varphi^x)(t)$. Thus for $x,y\in\R^n$, 
\[
e^{-\Phi(x)-\Psi(y)}=\int_0^{+\infty}e^{-\varphi^x(s)}ds\int_0^{+\infty}e^{-(\L\varphi)^y(t)}dt\le \frac{\pi}{2}e^{-\langle x,y\rangle}.
\]
This implies that $e^{-\Psi(y)}\le\frac{\pi}{2}e^{-\L\Phi(y)}$. Using \eqref{induction-func}, we get 
\[
\int_{\R^n}e^{-\Phi(x)}dx\int_{\R^n}e^{-\Psi(y)}dy\le \frac{\pi}{2}(2\pi)^n.
\]
that is
\[
\int_{0}^{+\infty}\int_{\R^n}e^{-\varphi(x,s)}dxds\int_{0}^{+\infty}\int_{\R^n}e^{-\L\varphi(y,t)}dydt\le \frac{\pi}{2}(2\pi)^n.
\]
 Adding this to the analogous bound for $s<0$ and using $\int_{s>0}\int e^{-\varphi(x,s)}dxds=\int_{s<0}\int e^{-\varphi(x,s)}dxds$, we conclude.

\begin{remark} Various $L_p$-versions of the functional Blaschke-Santal\'o's inequalities have been given (see for instance  \cite{HJM}). Also, Blaschke-Santal\'o's type inequality were established in the study of  extremal general affine surface area \cite{GHSW, Y, Hoe}. A consequence of the Blaschke-Santal\'o inequality was recently given in \cite{VY}. 
\end{remark}

\subsection{Lower bounds of the volume product of log-concave functions} 

Let $\varphi:\R^n\to\R\cup\{+\infty\}$ be convex.  The  {\it domain of $\varphi$} is $\dom(\varphi):=\{x\in\R^n; \varphi(x)<+\infty\}$. If $0<\int e^{-\varphi}<+\infty$, we define
the {\it functional volume product of $\varphi$} is
$$\VP(\varphi)=\min_z\int_{\R^n}e^{-\varphi(x)}dx\int_{\R^n}e^{-\L(\varphi_z)(y)}dy.$$
If $\varphi$ is even, this minimum is reached at $0$.
The following conjectures were proposed in \cite{FM3}.

\begin{conj}\label{mahler-func-conj} If $n\ge1$ and $\varphi:\R^n\to\R\cup\{+\infty\}$ is a convex function such that $0<\int e^{-\varphi}<+\infty$. Then 
\[
\int_{\R^n}e^{-\varphi(x)}dx\int_{\R^n}e^{-\L\varphi(y)}dy\ge e^n,
\]
with equality if and only if there is a constant $c>0$ and an invertible linear map $T$ such that 
\[e^{-\varphi(Tx)}=c\prod_{i=1}^ne^{-x_i}{\bf 1}_{[-1,+\infty)}(x_i).
\]
\end{conj}

\begin{conj}\label{mahler-func-conj-even} If $n\ge1$ and $\varphi:\R^n\to\R\cup\{+\infty\}$ is an even convex function such that $0<\int e^{-\varphi}<+\infty$. Then 
\[
\int_{\R^n}e^{-\varphi(x)}dx\int_{\R^n}e^{-\L\varphi(y)}dy\ge 4^n,
\]
with equality if and only if there exist a constant $c>0$, two complementary subspaces $F_1$ and $F_2$ and two Hanner polytopes $K_1\subset F_1$ and  $K_2\subset F_2$ such that for all $(x_1,x_2)\in F_1\times F_2$,
$$e^{-\varphi(x_1+x_2)}=ce^{-||x_1||_{K_1} }{\bf 1}_{K_ 2}(x_2).$$
\end{conj}

\begin{remark}
With a different duality for a convex function $\varphi$, another Blaschke-Santal\'o and inverse Santal\'o inequality were obtained in \cite{AS, FS}. 
 Another extension of Blaschke-Santal\'o inequality and of its functional form was considered in \cite{HoS}, where duality is combined with the study of inequalities related to  monotone non-trivial Minkowski endomorphisms.
\end{remark}

Partial results toward the proofs of these conjectures are gathered in the following theorem.
\begin{thm} Let $n\ge1$ and $\varphi:\R^n\to\R\cup\{+\infty\}$ be a convex function such that $0<\int e^{-\varphi}<+\infty$. Then
\begin{enumerate}
\item Conjecture \ref{mahler-func-conj} holds for $n=1$. It holds  also for all $n\ge 1$, if there exists an invertible affine map $T$ such that $\dom(\varphi\circ T)=\R_+^n$ and $\varphi\circ T$ is non-decreasing on $P$, in the sense that if $x_i\le y_i$ for all $1\le i\le n$, $(\varphi\circ T)(x_1,\dots, x_n)\le(\varphi\circ T)(y_1,\dots,y_n)$.
\item Conjecture \ref{mahler-func-conj-even} holds if $n=1$ or $n=2$.  It holds  also for all $n\ge 1$ if $\varphi$ is unconditional, in the sense that there exists an invertible linear map $T$ such that $(\varphi\circ T)(x_1,\dots,x_n)=(\varphi\circ T)(|x_1|,\dots,|x_n|)$ for all $(x_1,\dots,x_n)\in \R^n$.
\end{enumerate}
\end{thm}
\noindent(1) For $n=1$, Conjecture \ref{mahler-func-conj} was proved in two different ways in \cite{FM1, FM3}. The case of non-decreasing convex functions on the positive octant was also proved in \cite{FM3}. 

\noindent(2)  For unconditional convex functions on $\R^n$, Conjecture \ref{mahler-func-conj-even} was established in two different ways in \cite{FM2, FM3}, with the case of equality in \cite{FGMR}). In particular, this settles the general case  $n=1$.  For $n=2$, it was proved in \cite{FN}.

\begin{remark}\label{bodyfun}
There is a strong link between Conjectures \ref{mahler} and  \ref{conjcube} forconvex bodies and their functional counterparts Conjectures \ref{mahler-func-conj} and \ref{mahler-func-conj-even}. Indeed, as it was observed in \cite{FM3}, given a symmetric convex body $K$ in $\R^n$, if $\varphi_K(x)=\|x\|_K$, we get  $e^{-\L\varphi_K}={\bf 1}_{K^*}$, and   integrating on level sets, 
$\P(\varphi_K)=n!\P(K)$. Therefore, if Conjecture \ref{mahler-func-conj-even} holds for $\varphi_K$, then Conjecture \ref{conjcube} holds for $K$. Reciprocally, if Conjecture \ref{conjcube} holds in $\R^n$ for every dimension $n$ then, given an even, convex function $\varphi:\R^n\to\R\cup\{+\infty\}$, we can apply it 
in dimension $n+m$ to the convex sets
\[
K_m(\varphi)=\left\{(x,y)\in\R^n\times\R^m; \|y\|_\infty\le \left( 1-\frac{\varphi(mx)}{m}\right)_+\right\}.
\]
Using 
$$
\vol_{n+m}(K_m(\varphi))=\frac{2^m}{m^n}\int_{\R^n}\left( 1-\frac{\varphi(x)}{m}\right)_+^mdx
$$
and 
\[
K_m(\varphi)^*=\left\{(x,y)\in\R^n\times\R^m; \|y\|_1\le \inf_{\varphi(x')\le m} \frac{(1-\langle x,x'\rangle)_+}{1-\frac{\varphi(x')}{m}}\right\},
\]
it is proved in \cite{FM3} that when $m\to+\infty$, the inequality $\P(K_m(\varphi))\ge\frac{4^{n+m}}{(n+m)!}$ gives $\P(\varphi)\ge 4^n$. In a similar way, if Conjecture \ref{mahler-func-conj} holds in dimension $n+1$, given a convex body $K$ in $\R^n$ with  Santal\'o point at the origin, we apply it  
to  $\varphi:\R^n\times\R\to\R\cup\{+\infty\}$ defined, 
by 
\[
e^{-\varphi(x,t)}={\bf 1}_{[-n-1,+\infty)}(t){\bf 1}_{(t+n+1)K}(x)e^{-t}.
\]
Then the Legendre transform of $\varphi$ is
\[
e^{-\L\varphi(y,s)}={\bf 1}_{(-\infty,1]}(s){\bf 1}_{(1-s)K^*}(y)e^{(n+1)(s-1)},
\]
and 
\[
\P(\varphi)=\frac{(n!)^2e^{n+1}}{(n+1)^{n+1}}|K||K^*|.
\]
This proves that if Conjecture \ref{mahler-func-conj} holds for $\varphi$, then 
\[
\P(K)\ge\P(\Delta_n)=\frac{(n+1)^{n+1}}{(n!)^2},
\]
which is Conjecture \ref{mahler} for $K$. Lastly, as shown in \cite{FM3}, one can adapt the arguments  for even functions to prove that, given a convex function $\varphi:\R^n\to\R\cup\{+\infty\}$.  Conjecture \ref{conjcube} applied to a well chosen sequence of bodies $\Delta_m(\varphi)$ in dimension $n+m$ gives Conjecture \ref{mahler-func-conj} for $\varphi$  when $m\to+\infty$.
\end{remark}

 It was also proved in \cite{GM} that if $\VP(\varphi)$ is minimal, then  $\varphi$ has no positive Hessian at any point.
Asymptotic estimates hold too: in the even case, it was proved in \cite{KM} that for some constant 
$c>0$, one has for all even convex functions $\varphi$ and all $ n\ge 1$, $\VP(\varphi)\ge c^n$. This  was generalized to all convex functions in \cite{FM3}.

\subsection{Volume product and transport inequalities}

Maurey \cite{Mau} introduced the following property $(\tau)$: 
Let $\mu$ be a measure on $\R^n$ and $c:\R^n\times\R^n\to\R_+$ be a lower semi-continuous function (called a {\it cost function}); we say that the couple  $(\mu,c)$ satisfies {\it property $(\tau)$} if for any continuous and bounded function $f:\R^n\to\R$, defining 
$$Q_cf(y)=\inf_x\left(f(x)+c(x,y)\right)\mbox{ for $y\in \R^n$}, $$
one has
\[
\int_{\R^n}e^{-f(x)}d\mu(x)\int_{\R^n}e^{Q_cf(y)}d\mu(y)\le 1.
\]
Maurey \cite{Mau} showed that if $\gamma_n$ is the standard Gaussian probability measure on $\R^n$, 
with density $(2\pi)^{-n/2}e^{-|x|^2/2}$, 
and $c_2(x,y)=\frac{1}{2}|x-y|^2$ then as a consequence of the Prékopa-Leindler inequality, $(\gamma_n,\frac{c_2}{2})$ satisfies property $(\tau)$.

In \cite{AKM}, it was pointed out that the functional form of the Blaschke-Santal\'o inequality for the Legendre transform (Theorem~\ref{thm:bs-func}) is equivalent to an improved property $(\tau)$ for even functions:  we say that the pair $(\gamma_n,c_2)$ satisfies the {\it even property $(\tau)$} if for any even function $f$,    one has 
\begin{equation}\label{tau-prop}
\int_{\R^n}e^{-f(x)}d\gamma_n(x)\int_{\R^n}e^{Q_{c_2}f(y)}d\gamma_n(y)\le 1.
\end{equation}
This equivalence follows from the change of function: $\varphi(x)=f(x)+\frac{|x|^2}{2}$ and the fact that 
\[
-\L\varphi(y)=\inf_x\left(f(x)+\frac{|x|^2}{2}-\langle x,y\rangle\right)=Q_{c_2}f(y)+\frac{|y|^2}{2}.
\]
A direct proof of  \eqref{tau-prop} was then given by Lehec in \cite{Le}. And it follows from Remark~\ref{rk:lehec-centered} above, due to Lehec \cite{Le2}, that \eqref{tau-prop} also holds as soon as  $\int_{\R^n}xe^{-f(x)}d\gamma_n(x)=0$.

Moreover, as shown for example in Proposition 8.2 of \cite{GL},  there is a general equivalence between  property $(\tau)$ and symmetrized  forms of transport-entropy inequality. These transport-entropy inequalities were introduced by Talagrand \cite{Ta}, who showed that, for every probability measure $\nu$ on $\R^n$, one has
\begin{equation}\label{tala}
W_2^2(\nu,\gamma_n)\le 2 H(\nu|\gamma_n),
\end{equation}
where $W_2$ is the {\it Kantorovich-Wasserstein distance} defined by
\[
W_2^2(\nu,\gamma_n)=\inf\left\{\int_{\R^n\times\R^n}|x-y|^2d\pi(x,y); \pi\in\Pi(\nu,\gamma_n) \right\},
\]
where $\Pi(\nu,\gamma_n)$ is the set of probability measures on $\R^n\times\R^n$ whose first marginal is $\nu$ and second marginal is $\gamma_n$ and $H$ is the  {\it relative entropy } defined for $d\nu=fd\gamma_n$ by 
\[
H(\nu|\gamma_n)=-\int_{\R^n}f\log fd\gamma_n.
\]
Using this type of equivalence between property $(\tau)$ and transport-entropy inequalities, Fathi \cite{Fat} proved  the following symmetrized form of Talagrand's transport-entropy inequality: if $\nu_1$ (or $\nu_2$) is centered, in the sense that $\int xd\nu_1(x)=0$, then 
\begin{equation}\label{tal-fathi}
W_2^2(\nu_1,\nu_2)\le 2 (H(\nu_1|\gamma_n)+H(\nu_2|\gamma_n)).
\end{equation}
He showed actually that \eqref{tal-fathi} is equivalent to the functional form of Blaschke-Santal\'o's inequality (Theorem~\ref{thm:bs-func}). 
Applying \eqref{tal-fathi} to $\nu_1=\gamma_n$, one recovers Talagrand's inequality \eqref{tala}. In his proof, Fathi used  a reverse logarithmic Sobolev inequality for log-concave functions established in \cite{AKSW} under some regularity assumptions, removed later with a simplified proof in \cite{CFGLSW}.

In a similar way, Gozlan \cite{Go} gave equivalent transport-entropy forms of Conjectures 3 and 4 and of Bourgain-Milman's asymptotic inequality. This work was  pursued in \cite{FGZ}, where new proofs of the one-dimensional case  of Conjectures 3 and 4 are also provided.

\section{Generalization to many functions and bodies}\label{secMB}

The following intriguing conjecture was proposed  by Kolesnikov and Werner \cite{KoW}.
\begin{conj}\label{KW}  Let  $\rho:\R\to \R^+$  be increasing  and for $m\ge 2$, let $f_i:\R^n\to \R,$  $i=1, \dots, m$ be even   Lebesgue integrable  functions satisfying
$$
\prod_{i=1}^m f_i(x_i) \le \rho\left(\sum\limits_{1\le i<j\le m} \langle x_i, x_j\rangle \right) \mbox{ for all } x_1, \dots, x_m \in \R^n.
$$
Then
$$
\prod_{i=1}^m \int_{\R^n} f_i(x_i) dx_i \le \left(\int_{\R^n}\rho^{\frac{1}{m}}\left(\frac{m(m-1)}{2} |u|^2\right) du\right)^m.
$$
\end{conj}
Conjecture \ref{KW} was proved by Kolesnikov and Werner  when the functions $f_i$ are unconditional.
Observe that  Conjecture \ref{KW} is  a functional form of a new conjectured  Blaschke-Santal\'o inequality involving more than two convex bodies. 
Indeed, for $1\le i\le m$, let $K_i$ be starbodies, $f_i(x)=v_n(2\pi)^{-n/2}  e^{-\|x\|_{K_i}^2/2}$ and  $\rho(t)=e^{-t/(m-1)}$. 
 Since $\vol(K_i)=v_n(2\pi)^{-n/2} \int_{\R^n} e^{-\|x\|_{K_i}^2/2}dx, $
 we get from Conjecture \ref{KW}:
\begin{conj}\label{KWBS} Let $m \ge 2$ and let $K_1, \dots, K_m$ be  symmetric convex bodies in $\R^n$  such that
\begin{equation}\label{polKW}
\sum\limits_{1\le i<j\le m} \langle x_i, x_j\rangle \le \frac{m-1}{2}\sum_{i=1}^m\|x_i\|^2_{K_i}, \mbox{ for all } x_1, \dots, x_m \in \R^n,
\end{equation}
then
$
\prod_{i=1}^m \vol(K_i) \le \vol(B_2^n)^m.
$
\end{conj}
Conjecture  \ref{KW}  has been confirmed in \cite{KoW} for unconditional bodies; for $m\ge 3$, it was shown then that there is equality if and only if $K_i=B_2^n,$ for $i=1,\dots,m.$

This direction was further developed by Kalantzopoulos and Saroglou \cite{KaSa} who generalized the polarity condition (\ref{polKW}). For $2\le p\le m$ and  $ x_1, \dots, x_m \in \R^n,$ let
$$
{\mathcal S}_p(x_1,\dots, x_m)=\binom{m}{p}^{-1}\sum_{l=1}^n s_p(x_1(l), \dots, x_m(l)),   
$$
where $x_i=\sum_{l=1}^n x_i(l) e_l$ and $s_p$ is the elementary symmetric polynomial in $m$ variables of degree $p$. The case $p=2$ corresponds to the sum of scalar products, i.e. 
$${\mathcal S}_2(x_1,\dots, x_m)=\frac{2}{m(m-1)}\sum_{1\le i <j \le m} \langle x_i, x_j \rangle.$$ In \cite{KaSa} the following {\it $p$-Santal\'o conjecture} was proposed:
\begin{conj}\label{conKS} Let $2\le p\le m$ be two integers. If $K_1,\dots,K_m$ are symmetric convex bodies in $\R^n$, such that  
$$
{\mathcal S}_p(x_1,\dots, x_m) \le 1, \mbox{ for all   } x_i\in K_i,
$$
then $\prod_{i=1}^m \vol(K_i) \le \vol(B_p^n)^m,$ where $B_p^n$ is the unit ball of the $\ell^n_p$-norm. 
\end{conj}

Kalantzopoulos and Saroglou \cite{KaSa} were able to confirm Conjecture  \ref{conKS} when $p=m$, and in the case of unconditional convex bodies for all $p=2,\dots, m$. Moreover when $p=2$, it is enough to assume that only $K_3,\dots, K_m$ are unconditional. In all of those known cases, the conjectured inequality is actually sharp for $K_1=\dots=K_m=B_p^n$. A functional analog of Conjecture  \ref{conKS} was also proposed in \cite{KaSa}.

\section{Links to other inequalities}\label{linsec}

In this section we present just a sample of connections of the volume product to other inequalities in convex geometry. As before, we refer to the books  \cite{AGM1, AGM2, Ga, Ga2, Gru, Kol, Pi, Sc, Tom} and especially to the amazing diagram of connections of different open problems in convex geometry constructed by Richard Gardner in \cite[Figure 1]{Ga2}.

\subsection{Slicing Conjecture} Klartag \cite{Kl2} found a connection between a sharp version of Bourgain's slicing conjecture and Mahler's conjecture for general convex bodies (conjecture~\ref{mahler}-. The covariance matrix of a convex body $K$ in $\R^n$ is the $n\times n$ matrix ${\rm Cov}(K)$ defined by
$$
{\rm Cov}(K)_{i,j}=\frac{\int_K x_ix_jdx}{\vol(K)}- \frac{\int_K x_idx}{\vol(K)} \frac{\int_K x_jdx}{\vol(K)}
$$
The isotropic constant $L_K$ is  defined as $L_K^{2n}=\rm{det}({\rm Cov}(K)) \vol(K)^{-2}.$
It is well-known that $L_K$ is bounded from below by an absolute positive constant which is reached for ellipsoids. Bourgain’s slicing problem  asks whether for some universal constant $C>0$, one has $L_K \le C $ for every convex body $K$. The name {\it slicing conjecture} comes from the following  very interesting equivalent reformulation: is it true that for some universal $c>0$,  every convex body of volume one in $\R^n$ has  an hyperplane section with $(n-1)$-dimensional volume  greater than $c$? (see \cite{MiP} for other equivalent statements). The  boundedness of $L_K$ by an absolute constant is still  an open question.  Bourgain \cite{Bou} proved that $L_K\le  Cn^{1/4}$ up to a logarithmic factor, which was removed by Klartag \cite{Kl3}.
Chen \cite{Ch} proved that $L_K\le C_\varepsilon n^{\varepsilon}$ for every $\varepsilon>0$. Then,  Klartag and Lehec \cite{KlL} established
a polylogarithmic bound $L_K\le C\log^5n$, which was then further improved to $L_K\le C\log^{2.2}n$ by Jambulapati, Lee and Vempala \cite{JLV}. A {\it strong version of the slicing conjecture} asks the following:
\begin{conj}\label{strongBourgain}For any convex body $K$ in $\R^n$ one has
\begin{equation}\label{bourgainstron}
L_K \le L_{\Delta_n}=\frac{(n!)^{1/n}}{(n+1)^{\frac{n+1}{2n}}\sqrt{n+2}}.
\end{equation}
\end{conj}
Let $K$ be a local minimizer of the volume product among the set of all convex bodies in $\R^n$ endowed with the Hausdorff distance, then Klartag \cite{Kl2} was able to   prove that 
$${\rm Cov}(K^*) \ge (n+2)^{-2}{\rm Cov}(K)^{-1}.$$ 
Taking the determinant and raising to the power $1/n$, one gets
\begin{equation}\label{Klartagmin}
\frac{1}{n+2} \le L_K  L_{K^*}  \VP(K)^{1/n}.
\end{equation}
Thus combining   (\ref{Klartagmin}) and (\ref{bourgainstron}), 
$$
\frac{1}{n+2}  \le L_K  L_{K^*}  \VP(K)^{1/n} \le \frac{(n!)^{2/n}}{(n+1)^{\frac{n+1}{n}}(n+2)}\VP(K)^{1/n}.
$$
Thus, we proved the following theorem:
\begin{thm}  (Klartag)   
 The strong version  of Bourgain’s slicing conjecture given in  Conjecture \ref{strongBourgain}
implies Conjecture \ref{mahler} (Mahler’s conjecture) for general convex bodies.
\end{thm}

 In connection with his proof of the Bourgain-Milman inequality, Kuperberg  asked in \cite{Ku2}  whether the quantity
$$\frac{1}{\vol(K)\vol(K^*)} \int_K\int_{K^*}\langle x,y\rangle^2dxdy
$$
 is maximized for ellipsoids in the class of convex symmetric bodies $K\subset\R^n$.  Alonso-Guti\'errez \cite{AG} proved that this conjecture implies both the Blaschke-Santaló inequality and the hyperplane conjecture and that it holds true for  $B_p^n$, the unit ball of $\ell_p^n$,  for  $p\ge 1$. The connection to the hyperplane conjecture was also studied in \cite{Gi}. Kuperberg had not much hope for his conjecture and
 Klartag \cite{Kl2} showed that it is  false in high dimension, even in the case of unconditional bodies.

\subsection{Symplectic geometry and Viterbo's conjecture}   Artstein-Avidan, Karasev and Ostrover in \cite{AKO} discovered an amazing connection between the volume product and symplectic geometry. Let $(X,\omega)$ be a symplectic manifold: $X$ is a smooth manifold with a closed non-degenerate two-form $
\omega$. For instance, $(\R^{2n}, \omega_{st})$, where $\R^{2n}=\R^n_p\times \R^n_q$ and $\omega_{st}=\sum d p_i \wedge d q_i$. A core fact in  symplectic geometry states that symplectic manifolds have no local invariants (except the dimension). This, clearly,  makes the structure very different from  that of Riemannian manifolds.  The first examples of global symplectic invariants were introduced by Gromov \cite{Gro} and are known as Gromov’s “non-squeezing theorem”. Gromov's work inspired the introduction of global symplectic invariants - symplectic capacities - which may be seen as a way to measure the symplectic size of sets in $\R^{2n}$. More precisely, a {\it symplectic capacity} $c$ on $(\R^{2n},  \omega_{st})$ is a
mapping 
 $c: {\mathcal S}(\R^{2n}) \to \R_+$, 
where ${\mathcal S}(\R^{2n})$
is the set of all  subsets of $\R^{2n}$,  
which satisfies the following conditions
\begin{itemize}
\item Monotonicity: $c(U)\le c(V),$ for all $U\subset V.$
\item Conformality: $c(\phi(U))=|\alpha|c(U)$, for all diffeomorphism $\phi$ such that $\phi^* \omega_{st}=\alpha \omega_{st}.$
\item Normalization: $c(B^{2n}_2)= c(B^2_2 \times \R^{2(n-1)})= \pi.$
\end{itemize}

The following is the conjecture of  Viterbo \cite{Vi} for symplectic capacities of convex bodies. 
\begin{conj}\label{vitc} For any symplectic capacity $c$ and any convex body $\Sigma$ in $\R^{2n}$, one has
$$
\frac{c(\Sigma)}{c(B_2^{2n})} \le \left(\frac{\vol_{2n}(\Sigma)}{\vol_{2n}(B_2^{2n})}\right)^\frac{1}{n}.
$$
\end{conj}
Conjecture \ref{vitc} is  of isoperimetric type: indeed,   it claims that among all convex bodies  in $\R^{2n}$ of a given fixed volume, the Euclidean ball of the same volume has the maximal symplectic capacity. It is open even for $n=2$, but it holds for certain classes of convex bodies,
including ellipsoids  \cite{H} and  up to a universal multiplicative 
constant  \cite{AMO}. 
The following was proved in \cite{AKO}.
\begin{thm}
Conjecture \ref{vitc} implies Conjecture 2.
\end{thm}
More precisely, it was proved in \cite{AKO} that for any convex symmetric body $K\subset \R^n$, $c_{HZ}(K\times K^*)=4$,  where $c_{HZ}$ denotes the Hofer-Zehnder capacity, which is one of the important symplectic capacities.  This fact, together with Conjecture \ref{vitc} and the normalization property of $c_{HZ}$ immediately gives an affirmative answer to conjecture 2:
$$
\frac{4^n}{\pi^n} = \left(\frac{c_{HZ}(K \times K^*)}{c_{HZ}(B_2^{2n})}\right)^n \le \frac{\vol_{2n}(K \times K^*)}{\vol_{2n}(B_2^{2n})} = \frac{n!\vol_{2n}(K \times K^*)}{\pi^n}.
$$
We refer to \cite{AKO} and \cite{O} for more details on these  connections. The  connections of Conjecture 2 with symplectic geometry were further continued in \cite{ABKaS, BeKa, Ka, KaS}. In \cite{Ru}, Viterbo’s conjecture was connected  with
Minkowski versions of  worm problems, inspired by the well-known Moser worm
problem from geometry. For the special case of Lagrangian products, this relation provides further links to systolic Minkowski billiard inequalities and Mahler’s conjecture.

\subsection{Funk geometry}\label{Funk} A very interesting connection of the volume product with Funk geometry was recently discovered by Faifman \cite{Fa}. We refer to \cite{PT} for a detailed introduction to  Finsler manifolds and Funk geometry. We will remind a few of the most basic ideas.

A non-reversible Finsler manifold $(M,F)$ is a smooth manifold $M$ equipped with a smooth function $F$ on the tangent bundle of $M$ which, when restricted on any tangent space, is the gauge of some convex body. The crucial difference with Riemannian geometry is the lack of inner product. The tangent unit ball at a point $x\in M$ is denoted by $B_xM$ and consists of all vectors $v$ in the tangent space  $T_xM$ such that $F(x,v)\le 1.$ 

For a convex body $K$ in a  fixed affine space, the Funk metric on the interior of $K$ is given by  $B_xK =K$, i.e. at  any point $x$ in interior of $K$, the body $K$ with origin at $x$ is the unit ball. We done in the following way: Consider $x, y\in\inte(K)$ and let $R(x,y)$ be the ray starting at $x$ passing through $y$. Let $a(x,y)=R(x,y)\cap \partial K$, then the Funk metric, defined for $x\not=y\in \inte(K)$, is
$$
d_K^F(x,y)=\log\frac{|x-a(x,y)|}{|y-a(x,y)|},
$$
and $d_K^F(x,x)=0.$ The Funk metric is projective, i.e. straight segments are geodesics. The outward ball of radius $r>0$ and center  $z \in \inte(K)$ is 
$$
B^F_K(z,r)=\{x\in \inte(K): d_K^F(z,x) \le r\}=(1-e^{-r})(K-z)+z.
$$
 The Holmes-Thompson volume of $A \subset \inte(K)$ is defined as
$$
\vol_K^F(A)=\frac{1}{v_n}\int_A \vol(K^x) dx.
$$
 Asymptotically as $r\to 0$, the  volume of $B^F_K(z,r)$ behaves as  $v_n^{-1}\vol_{2n}(K\times K^z) r^n$. It was also shown in \cite{BBV} that for a strictly convex and smooth  body $K$, when $r\to+\infty$, the volume of $B^F_K(z,r)$ behaves as  $c_ne^{\frac{n-1}{2}r}\mathcal{A}(K,z)$, where $c_n>0$  depends only on $n$  and 
$\mathcal{A}(K,z)$ is the centro-affine
surface area of  $K$ defined by:

$$ \mathcal{A}(K,z)= \int_{\partial K}  \frac{\kappa^{1/2}_K(x)}{\langle x-z, n_K(x)\rangle^{(n-1)/2}} dx,
$$
where $\kappa_k(x)$ is Gauss curvature of $\partial K$ at point $x$ and $n_K(x)$ is outer normal vector, note that $\mathcal{A}(K,0)=\mathcal{A}(K)$.
The following duality relation for $\vol_K^F$, for centrally symmetirc $K$, is proved in \cite{Fa}:
$$
\vol_K^F(B^F_K(0,r))=\vol_{K^*}^F(B^F_{K^*}(0,r)).
$$
The existence of an analog of the Santal\'o point $s(K)$ of a convex body $K$ in the Funk geometry was proved in \cite{FaVW}: 
For any $r>0$, there is a unique point $s_r(K)\in\inte(K)$ that minimizes the Funk volume of $B_K^F(q,r)$. One has $s_r(K)=0$ for symmetric $K$ and $s_r(K)\to s(K)$ as $r\to 0$.  Let  $$
M_r(K)=v_n\vol_K^F(B_K^F(s_r(K),r))
$$
The following conjecture was proposed in \cite{Fa}:
\begin{conj}\label{conF} For all $r>0$, 
$M_r(K)$ is  maximal when $K$ is an ellipsoid.
\end{conj}
The limiting cases of Conjecture 
\ref{conF} are the Blaschke-Santal\'o inequality as $r\to 0$ and the centro-affine isoperimetric inequality as $r\to \infty$. Faifman \cite{Fa} was able to show that Conjecture 
\ref{conF} holds for unconditional bodies $K$. The idea of the proof  includes the generalization of the conjecture of K. Ball (see inequality (\ref{conjB})), namely
\begin{equation}
\int_K\int_{K^*}\langle x,y\rangle^{2j}dxdy \le  \int_{B_2^n}\int_{B_2^n}\langle x,y\rangle^{2j}dxdy,
\end{equation}
for all $j\in {\mathbb N}$, which  Faifman was able to confirm  for $K$ unconditional. 
 A lower bound for the quantity $M_r(K)$ was proposed in \cite{FaVW}: 
\begin{conj}\label{conFaVW} For  $r>0$,
$M_r(K)$
is minimized by simplices in general and by Hanner polytopes for symmetric bodies $K$.
\end{conj}
The limiting cases as $r\to 0$  of Conjecture \ref{conFaVW} for symmetric $K$  is Conjecture 2 and  
as $r\to +\infty$ is a conjecture of Kalai \cite{K} on the minimization of the flag number of $K$. Conjecture \ref{conFaVW} is proved in \cite{FaVW} for unconditional bodies and follows from an  interesting new inequality discovered in \cite{FaVW}  and proved for unconditional bodies
\begin{equation}\label{eq:lowerH}
\int_H\int_{H^*}\langle x,y\rangle^{2j}dxdy \le \int_K\int_{K^*}\langle x,y\rangle^{2j}dxdy, 
\end{equation}
where $H$ is a Hanner polytope in $\R^n$ and $j\in {\mathbb N}.$ The proof of  (\ref{eq:lowerH}) in \cite{FaVW} is based on the functional inverse Santal\'o inequality  \cite{FM2}.

\subsection{Geometry of numbers and isosystolic inequalities} The volume product is a standard tool in the geometry of numbers. The connection goes back to the theorem of Mahler \cite{Ma2} (see \cite{BM},  \cite{Gru}, Chapter 3 or \cite{Ev}) on the bound of the successive minima of a convex body and its dual.

Let us here present yet another connection of volume product with the  geometry of numbers and the systolic geometry  discovered by \'Alvarez Paiva, Balacheff and  Tzanev  \cite{APBT}.  

Minkowski’s first theorem in the geometry of numbers states that if $K$ is a symmetric convex body in $\R^n$ with $\vol(K) \ge 2^n$, then $K$ contains at least one non-zero integer point (in $\Z^n$). The symmetry assumption  is needed, as there are convex bodies $K$ of large volume containing the origin and no other integer point.  We know that such bodies must be  "flat" \cite{KL} and  \'Alvarez Paiva, Balacheff \cite{APB}  conjectured that the volume of their polars $K^*$ is not too small:
\begin{conj}\label{APconj} Let $K\subset \R^n$ be a convex body such that $\inte(K)\cap\Z^n=\{0\}$. Then  $\vol(K^*) \ge (n + 1)/n!$, with equality  if and only if 
$K$ is a simplex  with vertices in $\Z^n$ and no other integer  points than its vertices and $0$.\end{conj}
In \cite{APBT},  Conjecture \ref{APconj} was proved in $\R^2$ and an isomorphic bound  for $\vol(K^*)$ was given in all dimensions. Namely, for some absolute constant $c>0$, one has $\vol(K^*) \ge c^n(n + 1)/n!$ for any convex body  $K$ in  $\R^n$ such that $\inte(K)$ contains no integer point other than the origin. The proof of this fact in \cite{APBT} uses Bourgain-Milman inequality, and it is shown that this isomorphic version of Conjecture \ref{APconj} is actually equivalent to it.

Conjecture \ref{APconj} can be further generalized to a conjecture in systolic geometry.  We refer to \cite{APBT} for exact statements and definitions. We  mention here  a version of the conjecture in the language of Finsler geometry (see Section \ref{Funk}).  The Holmes-Thompson volume of a  Finsler manifold $(M,F)$ is defined as
$$
\vol_{HT}(M,F)=\frac{1}{v_n}\int_M \vol((B_xM)^*) dx.
$$
\begin{conj}\label{syst} For any Finsler metric $F$ on $\RP^n$, there exists a closed non-contractible   geodesic with length bounded by $\frac{(n!v_n)^{1/n}}{2}\vol_{HT}(\RP^n, F)^{1/n}$.
\end{conj}
 We remind that a set which can be reduced to one of its points  by a continuous deformation, is said to be contractible.  For $n=2$, Conjecture \ref{syst}  follows from the works of  Ivanov \cite{Iv1, Iv2}.  The next theorem was proved in \cite{APBT}:
\begin{thm}\label{systM}
Conjecture \ref{syst} implies Conjecture  2 for centrally symmetric bodies.
\end{thm}
The proof of Theorem \ref{systM} uses the Finsler  metric on  a convex symmetric body $K$ which coincides at each point with the norm  corresponding to $K$. By identifying  the points $x$ and $-x$ in $\partial K$, we obtain a length space (a space in which the intrinsic metric coincides with the original metric) on $\RP^n$. We denote this Finsler space by $(\RP^n,d_K)$. It turns out that one has
$$
\vol(\RP^n,d_K)=\frac{1}{v_n}\VP(K)
$$
and that the length of the systoles (the shortest noncontractible   geodesics) in $(\RP^n,d_K)$ is equal to $2$. Combining those commutations and assuming that Conjecture \ref{syst} holds, we get   from  Conjecture \ref{syst} a proof of Conjecture 2 for symmetric convex bodies in $\R^n$.

\vspace{0.8cm}

 \noindent Matthieu Fradelizi \\
Univ Gustave Eiffel, Univ Paris Est Creteil, CNRS, LAMA UMR8050 F-77447 Marne-la-Vall\'ee, France\\
E-mail address: matthieu.fradelizi@univ-eiffel.fr

\vspace{0.8cm}

\noindent Mathieu Meyer \\
Univ Gustave Eiffel, Univ Paris Est Creteil, CNRS, LAMA UMR8050 F-77447 Marne-la-Vall\'ee, France\\
E-mail address:  mathieu.meyer@univ-eiffel.fr

\vspace{0.8cm}

\noindent Artem Zvavitch \\
Department of Mathematical Sciences \\
Kent State University \\
Kent, OH 44242, USA \\
E-mail address: zvavitch@math.kent.edu

\end{document}